\documentclass[11pt]{amsart}
\usepackage{amsmath,amsfonts,amsthm}
\usepackage{amsmath}
\usepackage{graphicx,psfrag}  
\usepackage[psamsfonts]{amssymb}
\usepackage{amscd}
\usepackage{footnpag}
\usepackage[all,cmtip]{xy} 
\renewcommand{\thefootnote}%
{\fnsymbol{footnote}}

\title[H\"older foliations, revisited]{H\"older foliations, revisited 
}
\author{C. Pugh, M. Shub, and A. Wilkinson}
\thanks{Shub was partially supported by CONICET PIP 0801 2010-2012   and 
ANPCyT PICT 2010-00681.  Wilkinson was partially 
supported by the NSF.  The research in this paper was initiated while Pugh and Wilkinson were at Northwestern University.}
\address{Charles Pugh \\ Department of Mathematics \\University of Chicago\\ 5734 S. University Ave.\\
Chicago, Illinois 60637\\ USA \\ pugh@math.uchicago.edu {\em and  }
Department of Mathematics\\ University of California, Berkeley\\ 970 Evans Hall $\#$3840\\Berkeley, CA 94720 USA pugh@math.berkeley.edu}
\address{Michael Shub \\ CONICET, IMAS, Universidad de Buenos Aires, Buenos Aires, Argentina \\\emph{and} Department of Mathematics \\ The CUNY Graduate Center\\ 365 Fifth Avenue, Room 4208 \\ New York, NY 10016 \\  shub.michael@gmail.com}
\address{Amie Wilkinson \\ Department of Mathematics \\University of Chicago\\ 5734 S. University Ave.\\
Chicago, Illinois 60637\\ USA,  wilkinso@math.uchicago.edu}
\date{\today}
\theoremstyle{plain}

\def\title{\em}

\def\Bbb{\bf}

\def\cW{\mathcal{W}}

\def\cF{\mathcal{F}}

\def\cG{\mathcal{G}}

\def\cH{\mathcal{H}}

\def\transverse{\,\raise2pt\hbox to1em{\hfil$\top$\hfil}\hskip -1em \hbox
to1em{\hfil$\cap$\hfil}\,}

\newcommand\RR{{\mathbb R}}

\newlength{\figboxwidth} 

\theoremstyle{plain}
\newtheorem{Thm}{Theorem}
\newtheorem{Lemm}[Thm]{Lemma}

\newtheorem{Propn}[Thm]{Proposition}
\newtheorem{Add}[Thm]{Addendum}
\newtheorem*{Thm*}{Theorem}
\newtheorem*{Lemm*}{Lemma}
\newtheorem*{Propn*}{Proposition}
\newtheorem*{Cor*}{Corollary}
\newtheorem{Add*}{Addendum}

\newtheoremstyle{vThm}%
{}{}%
{\itshape}%
{0pt}{\bfseries}%
{}{ }%
{}%
\theoremstyle{vThm}

\newtheoremstyle{vThm*}%
{}{}%
{\itshape}%
{-3pt}{\bfseries}%
{}{ }%
{\thmnote{#3}}%
\theoremstyle{vThm*}
\newtheorem*{nThm*}{}

\newcounter{remark}
\theoremstyle{remark}
\newtheorem{Remark}[remark]{\bf Remark}
\newtheorem*{Rmk}{\bf Remark}

\newtheorem*{Ups}{\bf Upshot}

\newtheorem*{Defn}{\bf Definition}

\newcommand{\abs}[1]{\lvert#1\rvert}
\newcommand{\norm}[1]{\lVert#1\rVert}

\DeclareMathSymbol
{\rightrightarrows}
{\mathrel}{AMSa}{"13}

\DeclareSymbolFont{Csc}{U}{zcsc}{m}{n}
\DeclareMathSymbol{\bigsquareunion}{\mathop}{Csc}{"7C}
\DeclareMathSymbol{\bigintersection}{\mathop}{Csc}{"3C}
\DeclareMathSymbol{\bigunion}{\mathop}{Csc}{"3E}

\begin{document}

\large

\begin{abstract}
We investigate transverse H\"older regularity of some canonical leaf conjugacies in normally hyperbolic dynamical systems and   transverse H\"older regularity of some invariant foliations.  Our results validate   claims made elsewhere in the literature.
\end{abstract}

\maketitle

\section{Introduction}

A foliation $\mathcal{F}$ that is normally hyperbolic and plaque expansive with respect to a diffeomorphism $f$ is structurally stable in the following sense. For each $C^{1}$ small perturbation $g$ of $f$ there is a   $g$-invariant foliation $\mathcal{F}_g$ and a homeomorphism $\mathfrak{h}_{g} : M \rightarrow  M$ sending the original foliation equivariantly to $\mathcal{F}_g$.  Restricted to each $\mathcal{F}$-leaf, $\mathfrak{h}_{g}$ is $C^{1}$.    

In \cite{DK1} and \cite{DK2}  Damjanovi\'c  and   Katok assert that in the context of perturbations of Anosov actions (such as time one maps of Anosov flows) the homeomorphism $\mathfrak{h}_{g}$ can be chosen to satisfy a H\"older condition.  This does not follow from the standard fact proved by Anosov in \cite{A} and  the first two authors in \cite{HPS77} that the relevant $g$-invariant foliations are tangent to H\"older plane fields.  It is a subtler issue, and is a consequence of Theorem~A below.  

In \cite{IN11}  Ilyashenko and Negut treat the case of skew products, such as perturbations of an Anosov   diffeomorphism cross the identity.  We generalize their result to uniformly compact laminations in Theorem~B.    See Section~\ref{s.AB} for statements of Theorems~A and B.
The general question of when $\mathfrak{h}_{g}$ is H\"older remains open.

 \section{Background}
 \label{s.background}

A foliation $\mathcal{F}$ of a manifold $M$  is a division of $M$  into disjoint submanifolds called leaves of the foliation with the following properties.

\begin{itemize}

\item
Each leaf is connected although it need not be a closed subset of the manifold.

\item  The leaves all have the same dimension, say $c$.

\item  For each point $p$ in the manifold there exists a homeomorphism $\varphi $ from  $D^{c} \times  D^{m-c}$ onto a neighborhood of $p$ that carries each  $D^{c} \times  y$ to a subset of the leaf  containing $\varphi (0, y)$.  

\end{itemize}

$D^{c}$ is the open $c$-dimensional disc, $m$ is the dimension of the manifold $M$, and $m-c$ is the transverse dimension of the foliation.  Such a $\varphi $ is a \textbf{foliation box} and $\varphi (D^{c} \times y)$ is a \textbf{plaque} of the foliation.

\begin{Defn}
The \textbf{leaf topology} on $M$ is generated by the plaques in an atlas of foliation boxes.  It is denoted by $(M,\mathcal{F})$. The leaves are the connected components of $(M,\mathcal{F})$.  The \textbf{leaf space} is the set of leaves.  It is denoted by $M/\mathcal{F}$.
\end{Defn}

 Each leaf is a $c$-dimensional manifold covered by plaque coordinate neighborhoods.   The leaf topology $(M, \mathcal{F})$ projects to the discrete topology on $M/\mathcal{F}$.

The foliation $\mathcal{F}$ is $\pmb{f}$\textbf{-invariant} by  $f : M \rightarrow  M$ if $f$ permutes its leaves.  That is,
$$
\begin{CD}
 M @>\text{\normalsize
$\qquad 
f \qquad$}>> 
M
\\
@V\text{\normalsize$
\pi 
$}VV @VV\text{\normalsize$
\pi 
$}V
\\
M/\mathcal{F} @>\text{\normalsize$\qquad 
f  
\qquad$}>>
M/\mathcal{F}
\end{CD}
$$
commutes where $\pi $ projects the point    $p \in M$ to the leaf $\mathcal{F}(p)$ containing it.

\begin{Defn}
A \textbf{leaf conjugacy} from an $f$-invariant foliation $\mathcal{F} $ to a $g$-invariant foliation $\mathcal{G} $ is a homeomorphism $\mathfrak{h} : M \rightarrow  M$ sending $\mathcal{F} $-leaves to $\mathcal{G} $-leaves equivariantly in the sense that
$$
\begin{CD}
 M/\mathcal{F}  @>\text{\normalsize
$\qquad 
f \qquad$}>> 
M/\mathcal{F} 
\\
@V\text{\normalsize$
\mathfrak{h}
$}VV @VV\text{\normalsize$
\mathfrak{h}
$}V
\\
M/\mathcal{G}  @>\text{\normalsize$\qquad 
g 
\qquad$}>>
M/\mathcal{G} 
\end{CD}
$$
commutes.  In other words, $\mathfrak{h}(f(\mathcal{F}(p))) = g(\mathcal{G}(\mathfrak{h}(p)))$.
\end{Defn}

A foliation is smooth if there exists a covering of the manifold by foliation boxes, each of which is a    diffeomorphism.  Smooth foliations are studied widely in differential topology, but in    dynamics the naturally occurring foliations are only partially smooth.  In \cite{HPS77} the term ``lamination'' is used for this kind of foliation.  Here we suggest   revised terminology.  

\begin{Defn}
A foliation  is  \textbf{regular} 
if the manifold can be covered by foliation boxes   $\varphi  = \varphi (x,y)$ such that $\partial \varphi (x,y) /\partial x$ exists, is nonsingular, and depends continuously on $(x,y) \in D^{c} \times  D^{m-c}$.  A leaf conjugacy between regular invariant foliations is \textbf{regular} if its restriction to each leaf is $C^{1}$, non-singular, and these leaf derivatives are continuous on $M$.
\end{Defn}

The leaves of a regular foliation $\mathcal{F}$ of $M$ are $C^{1}$ and are assembled $C^{1}$-continuously.  The vectors tangent to its leaves   form a continuous subbundle  $T\mathcal{F} \subset TM$.   It is also natural to speak of a foliation being $C^{r}$ regular for $r > 1$.  Its leaves are $C^{r}$ and are assembled $C^{r}$-continuously.

If $\mathcal{F}$ is a regular foliation which is invariant by a diffeomorphism $f$ then  the tangent map $Tf : TM \rightarrow  TM$ sends $T\mathcal{F}$  isomorphically to itself and the diagram 
$$
\begin{CD}
 T\mathcal{F} @>\text{\normalsize
$\qquad 
Tf \qquad$}>> 
T\mathcal{F}
\\
@V\text{\normalsize$
\pi 
$}VV @VV\text{\normalsize$
\pi 
$}V
\\
M @>\text{\normalsize$\qquad 
f 
\qquad$}>>
M
\end{CD}
$$
commutes. (Here $\pi $ is the projection $TM \rightarrow  M$.)

\begin{Defn}
An $f$-invariant regular foliation $\mathcal{F}$ of a compact manifold $M$   is \textbf{normally hyperbolic}   if the tangent bundle of $M$ splits as a direct sum of continuous subbundles
$$
TM = E^{u} \oplus E^{c} \oplus E^{s}
$$
such that $Tf$ carries $E^{u}$, $E^{c} =T\mathcal{F}$,  and $E^{s}$ to themselves isomorphically, and for some Riemann structure on $TM$  we have
$$
T^{s}_{p}f < 1 < T^{u}_{p}f \quad \textrm{and} \quad T^{s}_{p}f < T^{c}_{p}f < T^{u}_{p}f  
$$
for all $p \in M$.  This is a shorthand expression where $T^{u}f$, $T^{c}f$, $T^{s}f$ are the restrictions of $Tf$ to the subbundles $E^{u}$, $E^{c}$, $ E^{s}$, and for linear transformations $A, B$   we write $A < B$ and $A < c < B$ to indicate
$$
\norm{A} < \pmb{m}(B) = \norm{B^{-1}}^{-1}  \quad \textrm{and} \quad 
\norm{A} < c < \pmb{m}(B) \ .
$$
 $\pmb{m}(B)$ is the \textbf{conorm} of $B$, the infimum of $\abs{B(u)}$ as $u$ varies over the unit vectors in the domain of $B$.
 \end{Defn}

 If we want to be more precise then  we choose continuous functions $\mu ,\nu , \widehat{\nu }, \widehat{\mu } :M \rightarrow  (0,1)$   and $\gamma , \widehat{\gamma } : M \rightarrow  (0, \infty )$ bracketting $T^{s}f$, $T^{c}f$, $T^{u}f$ in the sense that
 \begin{alignat*}{2}
 \mu (p) &< T^{s}_{p}f &&<  \nu (p)
\\
\gamma (p) &<T^{c}_{p} f &&<  (\widehat{\gamma }(p))^{-1}
\\
(\widehat{\nu }(p))^{-1} &<T^{u}_{p}f &&<  (\widehat{\mu }  (p))^{-1}\end{alignat*}
and we choose them so that $\nu < \gamma  < \widehat{\gamma }^{-1} < \widehat{\nu }^{-1}$.

A central result in \cite{HPS77}  concerns perturbations of a  normally hyperbolic  foliation $\mathcal{F}$.  

\begin{Thm}{\bf (Foliation Stability)}
\label{t.FST}
If   $\mathcal{F}$ is normally hyperbolic and plaque expansive (see below) then it is structurally stable in the following sense.  For each diffeomorphism $g$ that $C^{1}$-approximates $f$ there exists a unique  $g$-invariant foliation $\mathcal{F}_g$ near $\mathcal{F}$.  The foliation $\mathcal{F}_g$  is  normally hyperbolic, plaque expansive, and $(f, \mathcal{F})$ is canonically leaf conjugate to $(g, \mathcal{F}_g)$ by a homeomorphism $\mathfrak{h}  : M \rightarrow  M$.
\end{Thm}

See Section~\ref{s.canonical} for more details about $\mathfrak{h}$ and   clarification of the word ``canonically.''

\begin{Defn}
An $f$-invariant foliation is \textbf{plaque expansive} if there exist a plaquation (see below) $\mathcal{P}$ of $\mathcal{F}$ and a $\delta  > 0$ such that any two $\delta $-pseudo orbits of $f$ that respect  $\mathcal{P}$  and $\delta $-shadow each other necessarily belong to the same plaques of $\mathcal{P}$. 
\end{Defn}

\begin{Rmk}
A fundamental open question   is whether normal hyperbolicity implies plaque expansivity.  There are cases in which the implication is known, namely  
\begin{itemize}

\item
if the foliation $\mathcal{F}$ is of class $C^{1}$   \cite{HPS77},

\item
or as shown by Carrasco, if the leaves of $\mathcal{F}$ have uniformly bounded leaf volume \cite{Pablo},

\item
or as shown by Chillingworth and Hertz, Hertz and Ures, if $T^{c}f = Tf|_{T\mathcal{F}}$ is an isometry \cite{Chillingworth}, \cite{Federico},

\item or, as shown by Hammerlindl, if the  foliations $\cW^u$ and $\cW^s$ tangent to $E^u$ and $E^s$ are ``undistorted,'' for instance if $M$ is the $3$-torus \cite{Andy}.\footnote{The  definitions in the present paper are pointwise, not absolute as in \cite{Andy}, and do not imply  undistortedness.}

\end{itemize}
\end{Rmk}

 Here is a  more detailed description of plaque expansivity.  In the first place it generalizes the concept that the map $f$ is   expansive, meaning   there is a $\delta  > 0$ such that  for   any distinct orbits $(f^{n}(x))$ and $(f^{n}(y))$ there exists a $k \in \mathbb{Z}$ with 
 $$
 d(f^{k}(x), f^{k}(y)) > \delta  \ .
 $$
 ($d$ is a fixed metric on $M$.)
 In fact, if one considers the foliation of $M$ by its own  points then  leaves are points, plaques are points, and the two concepts coincide.

 The   idea is to replace   orbits of points by   orbits of plaques.  This is not quite possible because $f$ need not send plaques to plaques.  The $f$-image of a plaque may need to be adjusted (shrunk, stretched, or slid slightly along its leaf) in order to produce a new plaque of comparable size. 
 
 Formally, a 
   \textbf{plaquation} of $\mathcal{F}$ results from a choice of finitely many foliation boxes $\varphi  : D^{c} \times  D^{m-c} \rightarrow  M$ such that the corresponding half size foliation boxes $\varphi (\frac{1}{2}D^{c} \times  \frac{1}{2}D^{m-c})$  cover $M$.  The plaquation $\mathcal{P}$ consists of the unit  size plaques $\varphi (D^{c} \times  y)$.    They cover the leaves of $\mathcal{F}$ in a uniform fashion.
   
   A $\delta $-pseudo orbit of $f$ is a bi-infinite sequence of points  $(x_{n})$ such that for each $n \in \mathbb{Z}$, $d(f(x_{n}), x_{n+1}) < \delta $.   It \textbf{respects the plaquation}  $\mathcal{P}$ if  $f(x_{n})$ and $x_{n+1}$ always belong to a common plaque in $\mathcal{P}$. Plaque expansivity requires there to be a $\delta  > 0$ such that  if   $(x_{n})$ and $(y_{n} )$ are $\delta $-pseudo orbits that respect $\mathcal{P}$ and have $d(x_{n}, y_{n}) < \delta $ for all $n \in \mathbb{Z}$ then  $x_{n}$ and $y_{n}$ always lie in a common plaque   $\rho _{n} \in \mathcal{P}$.

 Equivalently, plaque expansivity means there are a plaquation $\mathcal{P}$ and  a $\delta  > 0$ such that for any sequences of plaques  $(\rho _{n})$ and $(\sigma _{n})$ in $\mathcal{P}$ with $f(\rho _{n}) \cap \rho _{n+1} \not=  \emptyset$ and  $f(\sigma _{n}) \cap \sigma _{n+1} \not=  \emptyset$, either there exists an $n$ such that the minimum distance between $\rho _{n}$ and $\sigma _{n}$ exceeds $\delta $ or   $\rho _{n} \cap \sigma _{n} \not= \emptyset$.  
In short, either plaque   orbits spread apart to distance $ > \delta $ or the plaques overlap.

It is not hard to see that plaque expansivity is independent of the metric $d$ and the plaquation $\mathcal{P}$.  

\begin{Rmk}
More general than normal hyperbolicity of $f$ is \textbf{partial hyperbolicity}.  One assumes that $TM$ has a   $Tf$-invariant splitting $ E^{u} \oplus E^{c} \oplus E^{s}$  as above, but $E^{c}$ is not necessarily integrable.  In this paper our focus is on normal hyperbolicity.   
\end{Rmk}

\section{Theorems A and B}
\label{s.AB}

The leaf conjugacy $\mathfrak{h}$ in the Foliation Stability Theorem above is $C^{1}$ on leaves and the leaf derivative is transversely continuous, but what about general transverse regularity?  Although   $\mathfrak{h} $ is not usually transversely differentiable \cite{A}, a natural guess would be that it can be chosen to satisfy a   H\"older condition in the transverse direction.  This is consistent with a remark of J\"urgen Moser to the effect that \emph{every conjugacy (and invariant structure) occurring naturally in smooth dynamics is H\"older.} 

\begin{nThm*}{\bf{Theorem A.}}
Suppose that $f : M \rightarrow  M$ is normally hyperbolic  at  $\mathcal{F}$ and the bundles $E^{cu}$,  $E^{cs}$ are of class $C^{1}$.  (This implies that $E^{c} = T\mathcal{F}$ is   $C^{1}$ and therefore $f$ is plaque expansive $\mathcal{F}$.)  If $g$ $C^{1}$-approximates $f$   then the canonical leaf conjugacy $\mathfrak{h} $ in the Foliation Stability Theorem is biH\"older continuous.  So are the holonomy maps along the leaves of the $g$-invariant foliations.
\end{nThm*}

BiH\"older continuity means what it says: The map and its inverse are H\"older continuous.  See Section~\ref{s.Holder} for estimates of the H\"older exponents.   

\bigskip

Theorem~B    concerns laminations  --  foliations of compact sets.
As defined in \cite{HPS77} a \textbf{lamination} $\mathcal{L}$ of a compact set   $\Lambda \subset M$ is a family of disjoint submanifolds (``leaves''  of the lamination) whose union is $\Lambda $ and which are assembled in a $C^{1}$ continuous fashion.  That is, $\Lambda $ is covered by ``lamination boxes,''  where a lamination box is a map $\varphi  : D^{c} \times  Y \rightarrow  \Lambda $, $Y$ is a fixed compact set, $\varphi $ is a homeomorphism to a relatively open subset of $\Lambda $,  and  $\partial \varphi (x,y)/\partial x$ is nonsingular and continuous with respect to $(x,y) \in D^{c} \times  Y$.  The discs $D^{c} \times  y$ are sent to plaques in the leaves.  Normal hyperbolicity of a diffeomorphism at an invariant lamination  is defined in the obvious way: $T_{\Lambda }M$ has a partially hyperbolic  $Tf$-invariant splitting $E^{u} \oplus E^{c} \oplus E^{s}$ with $E^{c} = T\mathcal{L}$.   An example is the orbits of an Axiom~A  flow on a basic set, such as a solenoid.  As shown in \cite{HPS77} the invariant manifold theory and the Foliation Stability Theorem hold equally in the lamination case.    

A simple type of lamination arises from a skew product   diffeomorphism  $f : B\times Z \rightarrow B\times Z$ where  $B$ and $Z$ are compact manifolds, and 
$$
f(b,z) = (f_{0}(b), f_{1}(b,z)) \ .
$$
Let $\Lambda _{0}$ be a hyperbolic set for $f_{0}$ 
 with hyperbolic splitting $T_{\Lambda _{0}}B = E^{u} \oplus E^{s}$.
   Then $\Lambda  = \Lambda _{0} \times Z$ is a compact $f$-invariant set smoothly laminated by the compact manifolds $b \times  Z$.   

If the hyperbolicity of the base map $f_{0}$ dominates $\partial f_{1}(b,z)/\partial z$ then $f$ is normally hyperbolic.  Plaque expansiveness is automatic: Hyperbolicity of $f_{0}$  implies $f_{0}$-orbits separate, which implies $f$-orbits of leaves separate.  Thus, $f$ pseudo-orbits of plaques separate.  

  The following is the main result in \cite{IN11}, which was proved earlier in a somewhat more specific context by   Ni\c tic\u a and   T\"or\"ok \cite{NT98}.

\begin{Thm}  \cite{IN11}\label{t.IN}
For a normally hyperbolic skew product lamination  as above, assume that     $E^{u}$, $E^{s}$ are trivial  product  bundles.  Also assume that the hyperbolic set $\Lambda _{0} \subset B$  has local product structure with respect to $f_{0}$.    If $g$ $C^{1}$-approximates $f$ then the leaf conjugacy is H\"older  and the holonomy maps along the leaves of the corresponding $g$-lamination are H\"older.  
\end{Thm}

A special case of Theorem~\ref{t.IN} occurs when $Z$ is a single point.  The center unstable and center stable laminations are the unstable and stable laminations through the hyperbolic set $\Lambda _{0}$ of $f_{0}$.  The fact that their holonomy maps are H\"older was proved by   Schmeling and   Siegmund-Schultze in  \cite{SS}.  In particular, they showed that the stable and unstable foliations of an Anosov diffeomorphism have H\"older holonomy.    

\begin{Defn}
A lamination  whose leaves are compact and for which the leaf volume of the leaves is uniformly bounded is   a \textbf{uniformly compact lamination}.  
\end{Defn}

A skew product lamination is uniformly compact because its leaves are all the same, namely $b \times  Z$.

Let $\mathcal{L}$ be a normally hyperbolic lamination with splitting $E^{u} \oplus T\mathcal{L} \oplus E^{s}$.  In \cite{HPS77} and elsewhere it is shown that there exist unique $f$-invariant local laminations $\mathcal{W}^{u}$ and $\mathcal{W}^{s}$ tangent at $\Lambda $ to $E^{u}$ and $E^{s}$.  They are called the \textbf{strong} unstable and stable laminations, and are sometimes denoted as $\mathcal{W}^{uu}, \mathcal{W}^{ss}$.  In this paper we denote them as $\mathcal{W}^{u}, \mathcal{W}^{s}$.   Their leaves have plaques $W^{u}(p,r)$, $W^{s}(p,r)$ for $p \in \Lambda $; $f$ expands $W^{u}(p,r)$ across $W^{u}(f(p), r)$ and contracts $W^{s}(p, r)$ into $W^{s}(f(p), r)$.   
 The   \textbf{local center unstable manifold} and \textbf{local center stable manifold}  of a leaf $L \in \mathcal{L}$ are
 $$
 W^{cu}(L,r) = \bigcup_{p \in L} W^{u}(p,r) \quad \textrm{and} \quad W^{cs}(L,r) = \bigcup_{p \in L} W^{s}(p,r) \ .
 $$
 They are immersed but not necessarily embedded.  Their plaques are 
 $$
 W^{cu}(p,r) = \bigcup _{q \in W^{c}(p,r)}W^{u}(q,r) \qquad 
 W^{cs}(p,r) = \bigcup _{q \in W^{c}(p,r)}W^{s}(q,r) \ ,
 $$
 which depend continuously on $p$.

\begin{Defn}
A normally hyperbolic lamination $\mathcal{L}$ is \textbf{dynamically coherent} if its    center unstable and center stable plaques  intersect in sub-plaques  of the lamination.  More precisely,   if $z \in W^{cu}(p,r ) \cap W^{cs}(q ,r )$ then $z \in \Lambda $ and the intersection is an open subset of  $\mathcal{L}(z)$. 
\end{Defn}

\begin{Rmk}  
Dynamical coherence is automatic under the hypotheses of Theorem~A and Theorem~\ref{t.IN}.  
\end{Rmk}

\begin{nThm*}{\bf Theorem B.}
Suppose that $f : M \rightarrow  M$ is normally hyperbolic at a   dynamically coherent, uniformly compact lamination $\mathcal{L}$.  Then $\mathcal{L}$ is plaque expansive,   the leaf conjugacy in the Lamination Stability Theorem is biH\"older, and the leaf holonomies  are H\"older.  
\end{nThm*}

Of course Theorem~B includes  Theorem~\ref{t.IN}.  See Section~\ref{s.Holder} for estimates of the H\"older exponents.

\begin{Rmk}
The standard definition of dynamical coherence applies to foliations of the whole manifold, not to laminations of a subset.  One    assumes $f : M \rightarrow  M$ is partially hyperbolic and its center unstable and center stable subbundles integrate to  invariant   foliations.\footnote{A current, frequently used definition of dynamical coherence does not include this invariance property.  See \cite{BWcoherence}.} 
It follows that the leaves of these foliations intersect in a foliation at which $f$ is normally hyperbolic.   In contrast, the previous definition starts with a normally hyperbolic lamination and makes no assumption about integrability of the center unstable and center stable subbundles.  After all, these bundles are only defined at the laminated set $\Lambda $, so it may not make sense to  integrate them globally.  But what about the case that the lamination is a foliation of  $M$?
\end{Rmk}

\begin{Propn}
\label{p.newdc}
The two definitions of dynamical coherence are equivalent for foliations.
\end{Propn}

Before we give the proof of this proposition, we remark that, in the case where the lamination ${\mathcal {L}}$ is a foliation, the hypothesis of dynamical coherence can be dropped: it  follows automatically from normal hyperbolicity and uniform compactness.  This was recently proved by Bohnet in her PhD thesis (see Theorem 1.26 in \cite{Boh}). 

\begin{proof}[\bf Proof]
Suppose the   foliation $\mathcal{F}$  is dynamically coherent as defined above.  The global center unstable manifolds $W^{cu}(L)$ for $L \in \mathcal{F}$   are tangent to $E^{cu}$, and we claim they foliate $M$.  Suppose that $W^{cu}(L)$ intersects $W^{cu}(L^{\prime})$ at $p \in L$.  (Then   the intersection contains $W^{u}(p)$.)  Let $\rho$ and $ \rho ^{\prime}$ be plaques of $W^{cu}(L)$ and $W^{cu}(L^{\prime})$ at $p $.  The new definition of dynamical coherence implies that   $W^{cs}(p,r )$ meets these plaques in relatively open subsets of the $\mathcal{F}$-leaf through $p$, namely $L$.  Therefore $L \cap W^{cu}(L^{\prime}) $ is relatively open in $L$.      
  See Figure~\ref{f.newdc}.
\begin{figure}[htbp]
\centering
\includegraphics[scale=.60]{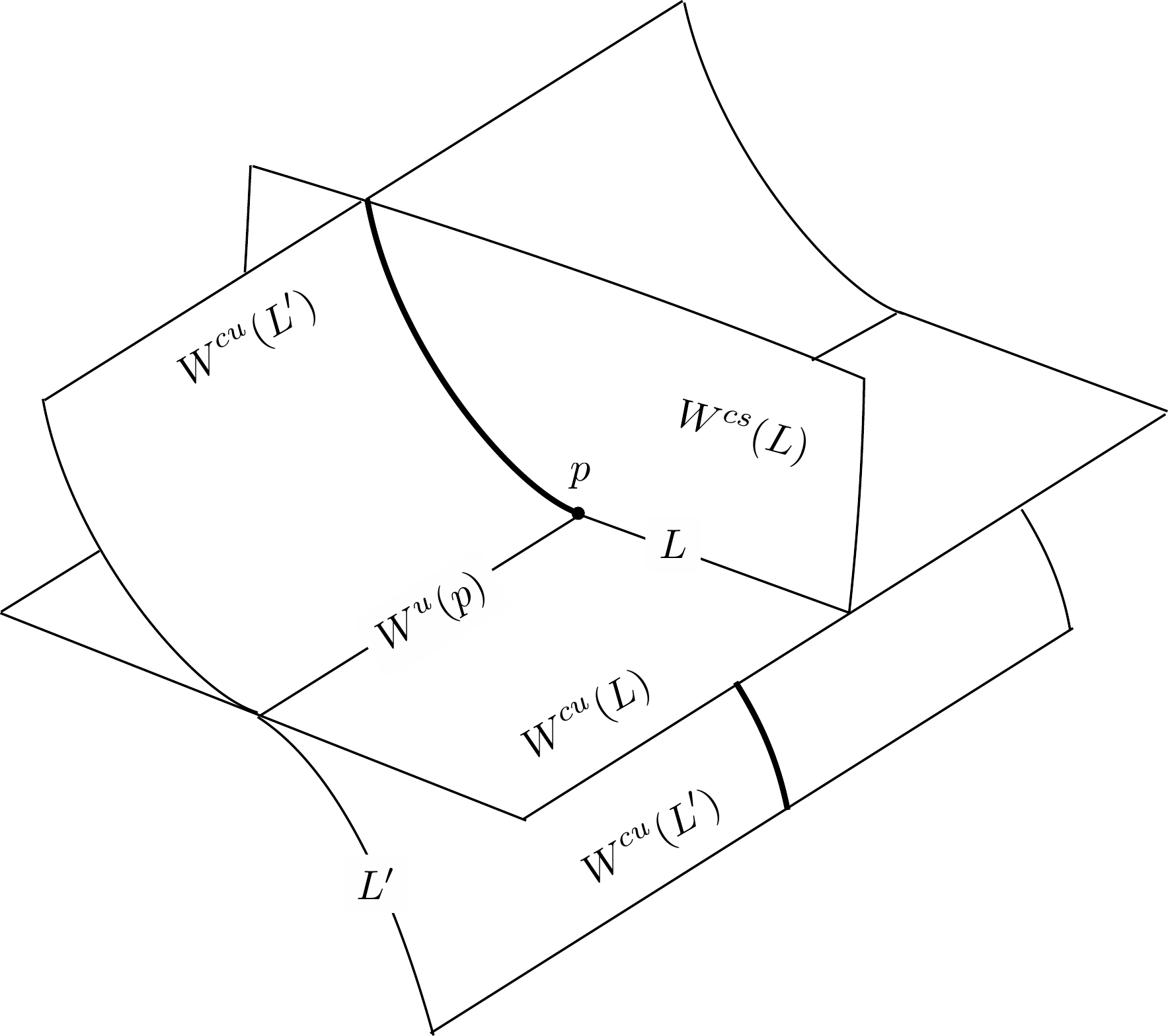}
\caption{Locally, the intersection $W^{cu}(L^{\prime}) \cap W^{cs}(L)$ is a plaque of the $\mathcal{F}$-leaf $L$ through $p$.  It is the dark curve and the new dynamical coherence condition implies that actually it    must equal $L$.}
\label{f.newdc}
\end{figure} 

With respect to the leaf topology  the intersection  is   closed in $L$.  Since the leaves are connected, $L \subset W^{cu}( L^{\prime})$.  Symmetrically, $L^{\prime} \subset  W^{cu}(L)$, so $W^{cu}(L) = W^{cu}(L^{\prime})$.  Moreover,   each $W^{cu}(L)$ is injectively immersed  --  it has no self-intersection  -- and thus
the global center unstable manifolds form a foliation that integrates $E^{cu}$.  
 Similarly the center stable manifolds form a foliation $\mathcal{W}^{cs}$ that integrates $E^{cs}$, so the new definition of dynamical coherence implies the standard one.

Conversely suppose   $f$ is partially hyperbolic and  $E^{cu}$, $E^{cs}$  integrate to    foliations $\mathcal{E}^{cu}$, $\mathcal{E}^{cs}$.  As shown in \cite{BWcoherence}, the foliations are $f$-invariant, the   intersection foliation $\mathcal{F}$  subfoliates each of them, and $f$ is normally hyperbolic at $\mathcal{F}$. Its normally hyperbolic splitting and its partially hyperbolic splitting coincide.   By uniqueness of the strong unstable leaves (see \cite{HPS77}) $\mathcal{W}^{u}$ subfoliates $\mathcal{E}^{cu}$.  Likewise $\mathcal{W}^{s}$ subfoliates $\mathcal{W}^{cs}$.  Thus $\mathcal{E}^{cu} = \mathcal{W}^{cu}$ and $\mathcal{E}^{cs} = \mathcal{W}^{cs}$, which means the intersection foliation for $\mathcal{W}^{cu}$ and $\mathcal{W}^{cs}$ is $\mathcal{F}$, as required by   the new definition of dynamical coherence.     
\end{proof}

Something of this survives for laminations.  
\begin{Propn}
\label{p.dclam}
If $\mathcal{L}$ is a normally hyperbolic, dynamically coherent lamination then its local center unstable manifolds meet in relatively open sets.  So do its local center stable manifolds.
\end{Propn} 
 \begin{proof}[\bf Proof]
The proof is the same as for foliations.  For it is local.
\end{proof}

 \section{The H\"older Exponents}
 \label{s.Holder}
 A map $f : X \rightarrow  Y$  from one metric space to another is $\theta $-H\"older if there is a constant $H$ such that for all $x, x^{\prime} \in X$ we have 
 $$
 d_{Y}(f(x), f(x^{\prime})) \leq Hd_{X}(x,x^{\prime}) ^{\theta } \ .
 $$
  Bunching and separation conditions among the bracketing functions $\mu , \nu , \gamma , \widehat{\gamma }, \widehat{\nu }, \widehat{\mu }$ described in Section~\ref{s.background} give estimates on the H\"older exponents of the normally hyperbolic summands, the leaf conjugacies, and the holonomies.      Recall that
   \begin{alignat*}{2}
 \mu (p) &< T^{s}_{p}f &&<  \nu (p)
\\
\gamma (p) &<T^{c}_{p} f &&<  (\widehat{\gamma }(p))^{-1}
\\
(\widehat{\nu }(p))^{-1} &<T^{u}_{p}f &&<  (\widehat{\mu } (p))^{-1}\end{alignat*}
for all $p \in M$, and   as functions
$$
0 < \mu <  \nu < 1 < \widehat{\nu }^{-1} < \widehat{\mu }^{-1} < \infty 
 \quad \textrm{and} \quad   \nu < \gamma  < \widehat{\gamma }^{-1} < \widehat{\nu }^{-1}   \ .
$$

  First we describe the known H\"older results when $f$ is $C^{2}$.     In \cite{A}, \cite{HPS77}, and elsewhere (e.g., in the work of Hasselblatt \cite{Hasselblatt}) it is shown that if $f$ is a $C^{2}$ partially hyperbolic diffeomorphism then the summands in its splitting are H\"older.  Specifically,  H\"olderness is implied by the bunching  conditions 
   \begin{itemize}
   \item[]  $E^{u}$ \;is $\theta $-H\"older  when $\widehat{\nu } < \widehat{\gamma }\mu ^{\theta }$.
      \item[ ]  $E^{s}$ \;is $\theta $-H\"older when $\nu  < \gamma \widehat{\mu }^{\theta }$.
         \item[ ]  $E^{cu}$ is $\theta $-H\"older when $\nu  < \gamma \mu ^{\theta }$.
            \item[ ]  $E^{cs}$ is $\theta $-H\"older when $\widehat{\nu } < \widehat{\gamma }\widehat{\mu} ^{\theta }$.
              \item[ ]  $E^{c}$ \;is $\theta $-H\"older when $\nu  < \gamma \mu ^{\theta }$ and $\widehat{\nu } < \widehat{\gamma }\widehat{\mu} ^{\theta }$.              
   \end{itemize} 
The notation is chosen so   the unstable conditions become the stable conditions by switching hats and non-hats.    As shown in \cite{Hasselblatt} and by Hasselblatt and Wilkinson in  \cite{HasselblattWilk} the    estimates are optimal in the $C^{2}$ case.  The  holonomy results for the strong foliations are similar: In \cite{PSW97} we show that   if $f$ is $C^{2}$ then
  \begin{itemize}
   \item  $\mathcal{W}^{u}$ \;has  $\theta $-H\"older  holonomy when $\widehat{\nu } < \widehat{\gamma }\mu ^{\theta }$.
  \item  $\mathcal{W}^{s}$ \;has  $\theta $-H\"older  holonomy when $\nu  < \gamma  \widehat{\mu }^{\theta }$.
   \end{itemize}

If $f$ is only  $C^{1}$ then it makes little sense to hope the summands are H\"older.  For they are $Tf$-invariant and  $Tf$ is only continuous.\footnote{In \cite{HasselblattWilk} it is shown that  if the holonomy is H\"older then the bundles are H\"older, correcting an assertion in \cite{PSW97}. The converse is false, as shown by Wilkinson in    \cite{W97}.}
But it does make sense to ask whether holonomy is H\"older.  For the invariant foliations are invariant by a $C^{1}$ diffeomorphism.  In  \cite{WLivsic} Wilkinson shows that  if $f$ is $C^{1}$ then 
   \begin{itemize}
   \item  $\mathcal{W}^{u}$ \;has  $\theta $-H\"older  holonomy when $\widehat{\nu}  < \widehat{\gamma} (\widehat{\nu} \mu )^{\theta }$.
  \item  $\mathcal{W}^{s}$ \;has  $\theta $-H\"older  holonomy when $\nu  < \gamma (\nu \widehat{\mu })^{\theta }$.
   \end{itemize}
We believe these  bunching    
conditions    are optimal for the strong holonomies but we have no proof.  We also have no proof that the other three types of holonomy (center unstable, center, and center stable) are H\"older in general.  What we    do prove in this paper are the following H\"older   assertions when      we perturb a normally hyperbolic  diffeomorphism whose invariant foliations are $C^{1}$ or when the center foliation is uniformly compact.

If $f$ is normally  hyperbolic, dynamically coherent,   its invariant foliations are of class $C^{1}$, and $g$ $C^{1}$-approximates $f$ then we will show 
   \begin{itemize}
     \item   $\mathcal{W}^{cu}_{g}$ has  $\theta^{2} $-H\"older  holonomy and the leaf conjugacy \\ $\mathfrak{h}^{cu} : \mathcal{W}^{cu} \rightarrow  \mathcal{W}^{cu}_{g}$  is  $\theta  $-H\"older when $\nu < \mu ^{\theta }$.

          \item  $\mathcal{W}^{cs}_{g}$ has  $\theta^{2} $-H\"older  holonomy and the leaf conjugacy \\ $\mathfrak{h}^{cs} : \mathcal{W}^{cs} \rightarrow  \mathcal{W}^{cs}_{g}$  is  $\theta  $-H\"older when $\widehat{\nu} < \widehat{\mu} ^{\theta }$.
                   \item  $\mathcal{W}^{c}_{g}$ has  $\theta^{2} $-H\"older  holonomy and the leaf conjugacy \\ $\mathfrak{h}^{c} : \mathcal{W}^{c} \rightarrow  \mathcal{W}^{c}_{g}$  is  $\theta  $-H\"older when 
                   $\nu < \mu ^{\theta }$ and $\widehat{\nu} < \widehat{\mu} ^{\theta }$.                     
   \end{itemize} 
(Recall that $\mathcal{W}^{c} = \mathcal{F}$ and $\mathcal{W}^{c}_{g} = \mathcal{F}_{g}$.)  See Section~\ref{s.A} for the proofs.
   
   If the center lamination of a $C^{1}$ normally hyperbolic, dynamically coherent diffeomorphism is uniformly compact then    we will show that
   \begin{itemize}

  \item  $\mathcal{W}^{c}$ \;has  $\theta $-H\"older  holonomy inside the center unstable leaves when $\widehat{\nu } < \widehat{\mu }^{\theta }$.

    \item  $\mathcal{W}^{c}$ \;has  $\theta $-H\"older  holonomy inside the center stable leaves when $\nu  < \mu ^{\theta }$.
    
     \item  $\mathcal{W}^{c}$ \;has  $\theta $-H\"older  holonomy   when $\widehat{\nu } < \widehat{\mu }^{\theta }$ and $\nu  < \mu ^{\theta }$.

   \end{itemize}
Furthermore, if $g$ $C^{1}$-approximates $f$  then    a canonical  leaf conjugacy $\mathcal{L} \rightarrow  \mathcal{L}_{g}$ is $\theta  $-H\"older when $\widehat{\nu } < \widehat{\mu }^{\theta }$ and $\nu  < \mu ^{\theta }$.     See Section~\ref{s.C} for the proofs.

\section{The Canonical Leaf Conjugacy}
\label{s.canonical}

The leaf conjugacy in the Foliation Stability Theorem is constructed in \cite{HPS77} as follows.  A smooth approximation $\widetilde{E}$ to $E^{u} \oplus E^{s}$ is chosen and   exponentiated into $M$.    This gives a smooth immersed tubular neighborhood $N(L, r)$ of each leaf $L \in \mathcal{F}$.  It is the union of tubular fibers  
$$
N(p,r) = \exp   \widetilde{E}(p,r)  
$$
for $p \in L$.  If $r$ is uniformly small and distinct points $p, q$ lie in a common plaque  $\rho $ of $L$ then   $N(p,r)$ and $N(q,r)$  are disjoint, but  for points $p, q$ in different plaques the fibers may meet badly.     

  Inside each $N(L,r)$ are local center unstable and center stable $f$-invariant manifolds that intersect in $L$.    Applying graph transform ideas to the diffeomorphism $g$ that approximates $f$  we get local center unstable and center stable manifolds for $g$ in $N(L,r)$.  Their intersection is $L_{g}$.    These  $L_{g}$ form a $g$-invariant  foliation $\mathcal{F}_g$.  It is the unique $g$-invariant foliation whose leaves approximate the leaves of $\mathcal{F}$.  

The leaf  conjugacy is a homeomorphism $\mathfrak{h} : M \rightarrow  M$ sending $L$ to $L_{g}$.  Specifically it sends $p \in L$ to the unique  point of the tubular fiber  $N(p,r)$ whose $g$-orbit can be closely shadowed by an $f$ pseudo-orbit that respects a fixed  plaquation $\mathcal{P}$ of $\mathcal{F}$.  In terms of what $\mathfrak{h}$ does to leaves, it is unique:  $L_{g} = \mathfrak{h}(L)$ is uniquely determined by $g$ and $L$.  However, as a point map $\mathfrak{h}$ depends on the choice of the  tubular neighborhood structure $\mathcal{N} = \{N(p,r)\}$.  A different choice of smooth approximation to $E^{u}\oplus E^{s}$ and  a different choice of smooth exponential map give a different tubular neighborhood structure $\mathcal{N}^{\prime}$, different tubular fibers, and consequently a different leaf conjugacy $\mathfrak{h}^{\prime}$.  

The relation between $\mathfrak{h}$ and $\mathfrak{h}^{\prime}$ is   simple.  They are homotopic by a homotopy $\mathfrak{h}_{t}$ that moves points a short distance in the plaques of $\mathcal{F}_g$.  For $\mathfrak{h}(p)$ and $\mathfrak{h}^{\prime}(p)$ belong to a common plaque $\rho $ of $\mathfrak{h}(L)$, so we can  draw the short geodesic $\gamma (t)$, $0 \leq  t \leq  1$,  from $\mathfrak{h}(p)$ to $\mathfrak{h}^{\prime}(p)$  and   project it to a path $\beta (t) $ in  $\rho $ using $\mathcal{N}$.  The   paths $\beta (t)$ give the homotopy.

So in this sense the leaf conjugacy is canonical: It is  unique as a leaf map,  it is unique as a plaque map, and as a point map it is unique up to a short plaque preserving homotopy.   

Finally, we weaken the smoothness of the tubular neighborhoods and speak also of laminations.

\begin{Defn}
A \textbf{tubular neighborhood structure} for a lamination $\mathcal{L}$ is a choice of $C^{1}$ discs $N(p,r)$ which are uniformly approximately tangent to $E^{u} \oplus  E^{s}$ at $\Lambda $ such that for each plaque $\rho $ in a plaquation of $\mathcal{L}$, the union of the tubular fibers $N(p,r)$ through points $p \in \rho $ forms a tubular neighborhood of $\rho $.  If the tubular neighborhood structure results from   exponentiating a smooth approximation to $E^{u} \oplus E^{s}$ it is called smooth. 
\end{Defn}  

The following   summarizes to what extent leaf conjugacies for foliations and laminations are canonical.   

\begin{Propn}
\label{p.canonical}
Suppose that $f$ is normally hyperbolic and plaque expansive at the dynamically coherent lamination $\mathcal{L}$.  Let $\mathcal{N}$ and $\mathcal{N}^{\prime}$ be tubular neighborhood structures for $\mathcal{L}$. 
 If $g$ $C^{1}$-approximates $f$ then the leaf conjugacies  $\mathcal{L} \rightarrow  \mathcal{L}_{g}$ corresponding to $\mathcal{N}$ and $\mathcal{N}^{\prime}$ are leaf canonical, plaque canonical, and  homotopic by a short homotopy in the plaques of $\mathcal{L}_{g}$.  
\end{Propn}

\section{Holonomy Intersection}
\label{s.intersection}
The following   lemma is used to deduce properties of the center foliation from facts about the center unstable and center stable foliations.

\begin{Lemm} {\bf (Holonomy Intersection)}
\label{l.Hproduct}
If transverse regular foliations  have $\theta $-H\"older holonomy then so does the intersection foliation.
\end{Lemm}

\begin{proof}[\bf Proof]
Let $\mathcal{F}$ and $\mathcal{G}$ be the transverse foliations.  The leaves of the intersection foliation $\mathcal{H}$ are the connected components of the intersections of leaves of $\mathcal{F}$ and $\mathcal{G}$.  Take  compact  smooth local transversals $\tau  $ and $ \tau ^{\prime}$ to $\mathcal{H}$ at $x$ and $x^{\prime}$ such that $x$ and $x^{\prime}$ belong to the same $\mathcal{H}$-leaf, $\mathcal{H}(x) = \mathcal{H}(x^{\prime})$.  The foliations $\mathcal{F}$ and $\mathcal{G}$ intersect $\tau $ in transverse foliations $\mathcal{F}_{\tau }$ and $\mathcal{G}_{\tau }$ of complementary dimensions.  Their leaves meet in points.  Likewise for $\tau ^{\prime}$.  

Let $h : \tau  \rightarrow  \tau ^{\prime}$ be an $\mathcal{H}$-holonomy map.  It arises from choosing a path $\gamma $ from $x$ to $x^{\prime}$ in   $\mathcal{H}(x)$, and then lifting   $\gamma $ to nearby $\mathcal{H}$-leaves.   Since $\mathcal{H}$-leaves are contained in $\mathcal{F}$- and $\mathcal{G}$-leaves, $h$   carries the leaves of $\mathcal{F}_{\tau }$ to leaves of $\mathcal{F}_{\tau ^{\prime}}$ and likewise for $\mathcal{G}$.  

Transversality of $\mathcal{F}_{\tau }$ and $\mathcal{G}_{\tau }$ implies there is a ``foliation triangle inequality'' for distance in $\tau $, namely if $y = \mathcal{F}_{\tau }(p) \cap \mathcal{G}_{\tau }(q)$ then 
$$
\frac{1}{D} \max \{ d_{\mathcal{F}_{\tau }}(p,y) , d_{\mathcal{G}_{\tau }} (y,q)\} \leq  
d_{\tau }(p,q) \leq D(d_{\mathcal{F}_{\tau }}(p,y) + d_{\mathcal{G}_{\tau }}(y,q))
$$ 
where $D$ is a constant determined by the foliations and $\tau $.  A similar statement can be found in Proposition 19.1.1 of Katok and Hasselblatt's book \cite{KH95}.   The  distances are measured along $\tau $ or the intersection leaves.   This does not use the fact that the intersection foliations are H\"older, but merely the fact that the angles between their leaves are bounded away from $0$.   
A similar triangle inequality holds at $\tau ^{\prime}$.  See Figure~\ref{f.FG}.
\begin{figure}[htbp]
\centering
\includegraphics[scale=.60]{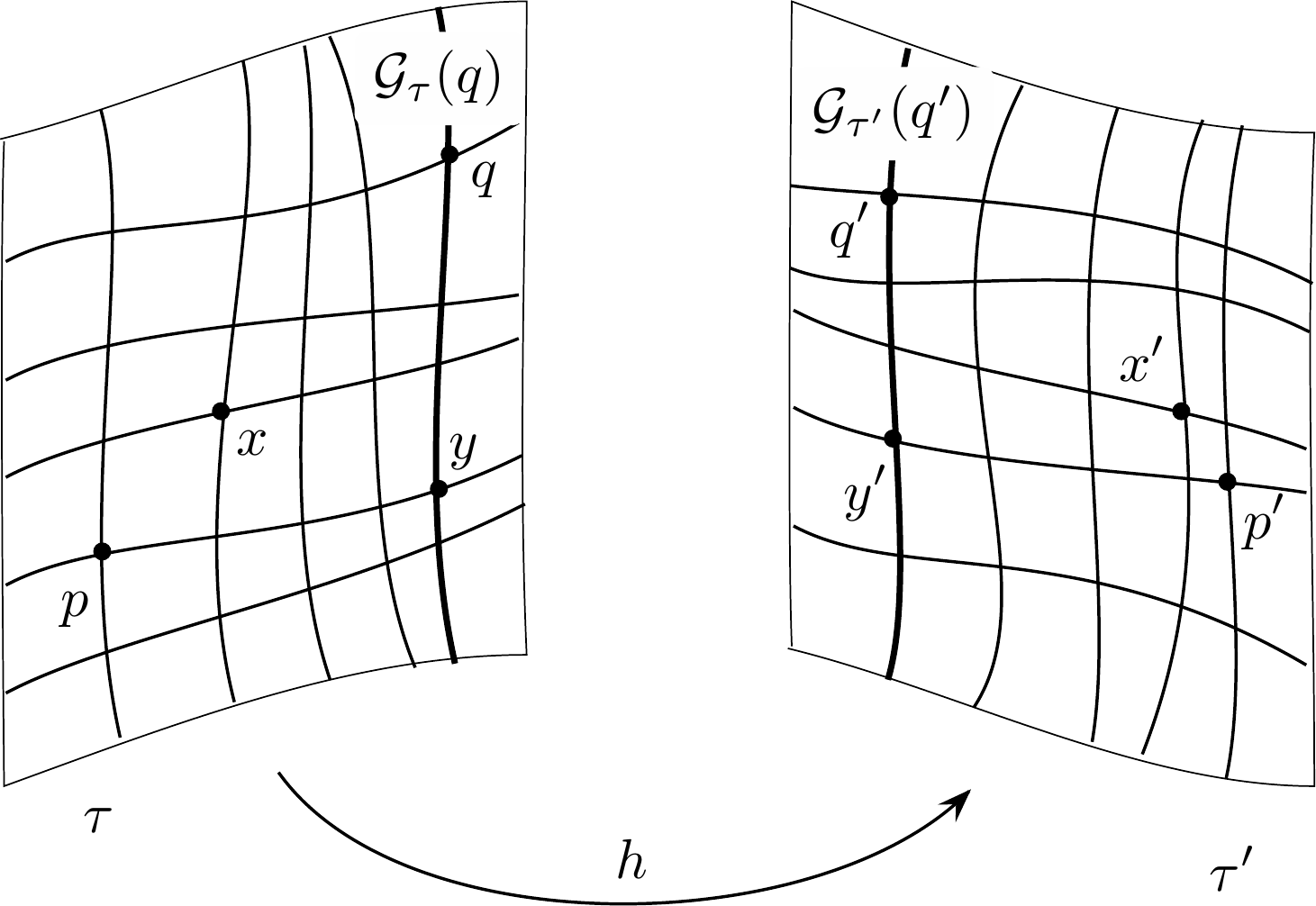}
\caption{Transversality of the intersection foliations on $\tau $ and $\tau ^{\prime}$ implies   modified triangle inequalities.  The restriction of the $\mathcal{H}$-holonomy map $h$  to $\mathcal{G}_{\tau }(q)$ is $\mathcal{F}$-holonomy.}
\label{f.FG}
\end{figure}

$\tau $ is not a transversal to $\mathcal{F}$.  Its dimension is wrong.  Rather,   $\mathcal{G}_{\tau }(q)$ is a transversal to $\mathcal{F}$ at $q \in \tau $ and $\mathcal{G}_{\tau ^{\prime}}(q^{\prime})$ is a transversal to $\mathcal{F}$ at $q^{\prime} =h(q) \in \tau ^{\prime}$.   The map $h$ restricted to $\mathcal{G}_{\tau }(q)$ is merely the $\mathcal{F}$-holonomy with respect to   $\gamma $.  Likewise for $\mathcal{G}$. These holonomy maps are H\"older.   Thus, if $h(p) = p^{\prime}$, $ h(q) = q^{\prime}$, $h(y) = y^{\prime}$ as above, and $d_{\mathcal{F}_{\tau } }(p,y)  \leq  d_{\mathcal{G}_{\tau } }(y,q)$ then 
\begin{equation*}
\begin{split}
d_{\tau ^{\prime}}(h(p),h(q)) & \leq  D^{\prime} (d_{\mathcal{F}_{\tau ^{\prime}}}(p^{\prime},y^{\prime}) + d_{\mathcal{G}_{\tau ^{\prime}}}(y^{\prime}, q^{\prime}))
\\
& \leq  D^{\prime}(Cd_{\mathcal{F}_{\tau }}(p,y)^{\theta } + Cd_{\mathcal{G}_{\tau } }(y,q)^{\theta })
\\
& \leq 2CD^{\prime}d_{\mathcal{G}_{\tau } }(y,q)^{\theta } \leq 2CD^{\prime}D^{\theta }d_{\tau }(p,q)^{\theta } 
\end{split}
\end{equation*}
shows that $h$ is $\theta $-H\"older.
\end{proof}

\begin{Rmk}
In the proof of the  Lemma~\ref{l.Hproduct} we did not need to know that all the holonomy maps of $\mathcal{F}$ and $\mathcal{G}$ are H\"older, only the ones that arise from $\mathcal{H}$-holonomy maps.  It is possible that $\mathcal{F}$ and $\mathcal{G}$ have other holonomy maps which fail to be H\"older.  This still permits the intersection foliation to be H\"older.   See Remark~\ref{k.goodbad} in Section~\ref{s.cautionary}.
\end{Rmk}

A slight sharpening of the   Lemma~\ref{l.Hproduct} replaces the assumption about $\mathcal{F}$ and $\mathcal{G}$ being H\"older by what was actually used in the proof, namely H\"olderness of the slice maps.   This also removes consideration of   irrelevant holonomy maps.  Likewise, the proof works just as well for laminations as for foliations.

\begin{Thm}
\label{t.intersection}
Suppose that $\mathcal{F}, \mathcal{G}$ are transverse laminations which intersect in a lamination $\mathcal{H}$. If     discs $\tau , \tau ^{\prime}$ are transverse to $\mathcal{H}$ and   $h :  {\tau } \rightarrow   {\tau ^{\prime}}$ is a holonomy map along $\mathcal{H}$   whose slice maps are uniformly $\theta $-H\"older then $h$ is $\theta $-H\"older.  
\end{Thm}

\begin{proof}[\bf Proof]  The slice maps of $h$ are its restrictions   to the slices $  \mathcal{F}_{\tau }(p)$ and $  \mathcal{G}_{\tau }(q)$.  It follows that $h$  sends $ \mathcal{F}_{\tau }(p) $ to $   \mathcal{F}_{\tau ^{\prime}}(h(p) )
$
  and similarly for $\mathcal{G}$.  
The inequality $d_{\tau ^{\prime}}(h(p), h(q)) \leq  2CD^{\prime}D^{\theta }d_{\tau }(p,q)^{\theta }$ has exactly the same proof.
\end{proof}

\section{Holonomy Versus Leaf Conjugacy}
\label{s.versus}

There is a natural relationship between   leaf conjugacies arising in \cite{HPS77} and   holonomies in certain skew products.  It lets us deduce H\"olderness of a leaf conjugacy in dimension $n$ from H\"olderness of  holonomy maps of the suspended foliation in dimension $n+1$.   This suspension strategy cuts our work in half. 

We consider compact manifolds $T, M$ and a skew product diffeomorphism $G : T \times  M \rightarrow  T \times  M$ covering the identity
$$
G(t,x) = (t, g_{t}(x)) \ .
$$
$g_{t} : M \rightarrow  M$ is the slice of $G$ over $t$.  In our application $T$ is the circle or the segment.  We   assume that for some $0 \in T$, 
\begin{itemize}

\item[(a)]
 $f = g_{0}$ is normally hyperbolic and plaque expansive at  a foliation $\mathcal{F} $ of $M$.
\item[(b)]
$G$  $C^{1}$-approximates the product diffeomorphism 
\\
$F (t,x) = (t, f (x))$.

\end{itemize}
  
Then $F $ is normally hyperbolic and plaque expansive at the product foliation $\mathcal{S}  = T \times  \mathcal{F}  $ whose leaves are products   $T\times L$ where $L$ is a leaf of  $\mathcal{F}$.  Also, each $g_{t}$ $C^{1}$-approximates $f $.  Fix a smooth  bundle $E \subset  TM$ complementary to $T\mathcal{F} $ and set $N(x,r) = \exp E(x,r)$.  Then $\mathcal{N} = \{N(x,r) : x \in M\}$ is a convenient $C^{1}$  family of smooth discs   transverse to the leaves  of $  \mathcal{F} $. Likewise $t \times  N(x,r)$ is a small smooth disc through $(t,x)$ in $T \times  M$  transverse to the leaves of the product foliation $\mathcal{S} $.    

We observe two things about these discs.  First, by Topogonov's Triangle Theorem \cite{CE}, there   exists a uniform  $r > 0$ such that   $N(\rho ,r) = \bigcup_{x \in \rho } N(x,r)  $ is a    tubular neighborhood of the  plaque  $\rho $.   Its natural parameterization is the $C^{1}$ diffeomorphism  $e_{\rho } : E(\rho ,r) \rightarrow N(\rho ,r)$ that sends  $(x,v)$ to $\exp_{x}(v)$.  Second,  $e_{\rho }$ has uniformly bounded distortion.  By this we mean that for some constant  $D$ and all plaques $\rho $, the map $e_{\rho }$ neither expands distance by more than a factor $D$, nor contracts it by less than a factor $1/D$.  This follows by further applications of the Toponogov Triangle Theorem.  In particular, local holonomy maps from one plaque to another along the fibers of $N(\rho ,r)$ have uniformly bounded distortion.  

  Then \cite{HPS77} implies two things, one about $G$ and the other about $g_{t}$.   
\begin{itemize}

\item[(c)]
There is a unique $G$-invariant foliation $\mathcal{S}_{G}$ near   $\mathcal{S} $, and there is an equivariant    leaf conjugacy $\mathfrak{h} : T \times  M \rightarrow  T \times  M$  which approximates the identity map and sends $\mathcal{S} $-leaves to $\mathcal{S}_{G}$-leaves.  In fact it    sends $(t,x) \in T \times  M$ to the unique point $(t,y) \in t \times N(x,r)$ whose $G$-orbit can be closely shadowed by an $F $ pseudo-orbit  that respects $\mathcal{S} $.  Modulo the choice of $\mathcal{N}$, $\mathfrak{h}$ is unique.
 
\item[(d)]
There is a unique $g_{t}$-invariant foliation $\mathcal{F}_{t}$ near   $\mathcal{F} $, and there is    an equivariant  leaf conjugacy $\mathfrak{h}_{t} : M \rightarrow M$ which approximates the identity map and sends $\mathcal{F} $-leaves to $\mathcal{F}_{t}$-leaves. In fact it sends $x \in M$ to the unique point $y \in N(x,r)$ whose $g_{t}$-orbit can be closely shadowed by an $f $ pseudo-orbit that respects   $\mathcal{F} $.  Modulo the choice of $\mathcal{N}$, $\mathfrak{h}_{t}$ is unique.
\end{itemize}

\begin{Defn}
The foliation $\mathcal{S}_{G}$ is the \textbf{suspension foliation} for $G$.   
\end{Defn}  

Consider transversals $0 \times  N(p,r)$ and $1 \times N(p,r)$ to an $\mathcal{S} $-leaf $T \times L$.   For $x  \in N(p,r)$   the straight line path $\sigma  : t \mapsto  (t,x )$ lies in the $\mathcal{S} $-leaf   containing $(0, x) $.  It   lifts to a nearby path $t \mapsto  (t,   \alpha (t,x ))$ in the $\mathcal{S}_{G}$-leaf through $(0,x )$ such that $\alpha (t, x ) \in   N(p,r)$ for $0 \leq  t \leq  1$.  The map $h_{p} : x  \mapsto \alpha (1,x )$ is a \textbf{suspension holonomy map}  for $G$.       It sends $N(p, r^{\prime})$ into $N(p, r)$ and  is independent of the choice of  $\alpha $ near $\sigma $.  The radius $r^{\prime}$ is less than $r$ so that $h_{p}(x)$ belongs to $N(p,r)$.  (Actually, one should define the suspension holonomy   as $H_{p} : (0, x) \rightarrow (1,h_{p}(x))$, but we abuse the concept for notational simplicity.)   See Figure~\ref{f.suspension} in which  $\sigma $ actually ``drops'' to $\alpha  $.  
  \begin{figure}[htbp]
\centering
\includegraphics[scale=.60]{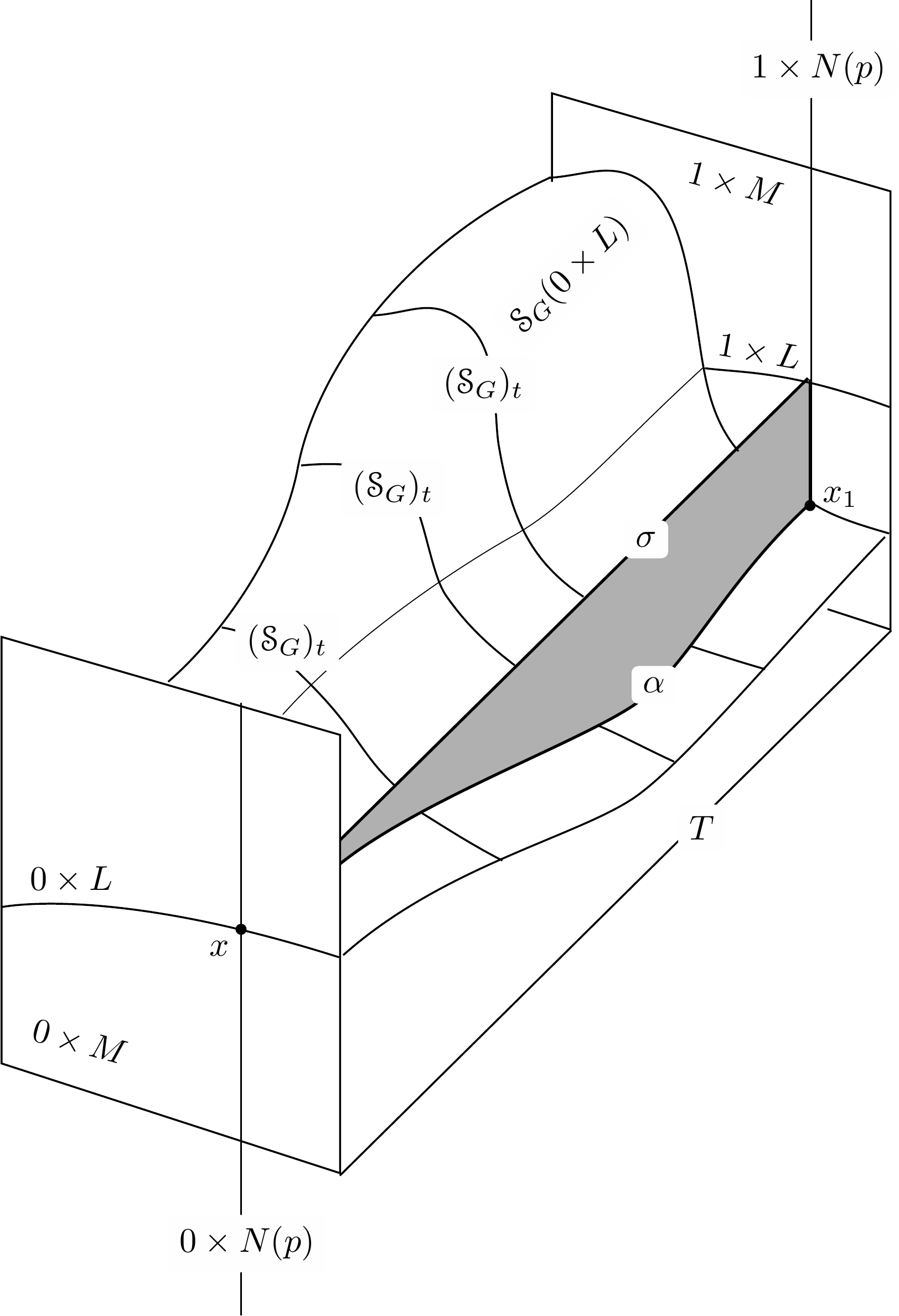}
\caption{The suspension foliation $\mathcal{S}_{G}$ and its suspension holonomy.  The $\mathcal{S}_{G}$-leaf through $0\times L$ is drawn.    The slice $[0,1] \times  N(p)$ between $\sigma $ and $\alpha  $ is shaded.  The holonomy sends $x$   to $x_{1}$.}
\label{f.suspension}
\end{figure}

 If $x \in N(p,r^{\prime}) \cap N(q,r^{\prime})$ then the straight line path $\sigma $ from $(0,x)$ to $(1,x)$ lifts to two nearby paths 
 $$
 (t, \alpha (t)) \quad \textrm{and} \quad (t, \beta (t))
 $$
in the same $\mathcal{S}_{G}$-leaf.   Thus $h_{p}(x)$ and $h_{q}(x)$ lie in the same $\mathcal{F}_{1}$-plaque.

\begin{Thm}\label{t.Amie1}   
The suspension  leaf conjugacies are related by 
$ \mathfrak{h}(t,p) = (t,\mathfrak{h}_{t}(p))$.  The $t$-slice of  $\mathcal{S} _{G}$ equals   $\mathcal{F}_{t}$, and $\mathfrak{h}_{t} $ is a canonical leaf conjugacy $\mathcal{F} \rightarrow  \mathcal{F}_{t}$.   The suspension  holonomy map satisfies
$$
h_{p}(x) = N(p,r) \cap \mathfrak{h}_{1}(\rho (x))
$$
where  $x \in N(p,r^{\prime})$ and $\rho (x)$ is its $\mathcal{F} $-plaque.  In particular,  $h_{p}(p) = \mathfrak{h}_{1}(p)$.  If the suspension holonomy maps $h_{p}$  are uniformly $\theta $-biH\"older then so is the leaf conjugacy $\mathfrak{h}_{1}$.
\end{Thm}

 \begin{proof}[\bf Proof]
 ``All this   follows naturally   from   the dynamical characterization    of  leaf conjugacy.''  Here are the details.
 According to (c), 
 $\mathfrak{h}(t,p)$ is the unique $(t,y) \in t \times  N(p,r)$ whose $G$-orbit $(t,g^{n}_{t}(y))$ is closely shadowed by an $F $ pseudo-orbit $(t_{n},p_{n})$.  Thus $(p_{n})$ is an $f $ pseudo-orbit that closely shadows the $g_{t}$-orbit of $y \in N(p,r)$ and respects $\mathcal{F} $.  By uniqueness in (d),  $y = \mathfrak{h}_{t}(p)$, i.e., $\mathfrak{h}(t,p) = (t,\mathfrak{h}_{t}(p))$. 
 
   The leaves of $\mathcal{S}_{G}$ approximate the product leaves and are transverse to the slice $t \times  M$, so the $t$-slice of  $\mathcal{S}_{G} $ is a $g_{t}$-invariant foliation   of $M$ that approximates $\mathcal{F} $.  By uniqueness it equals $\mathcal{F}_{t}$, and $\mathfrak{h}_{t}$ is the canonical leaf conjugacy $\mathcal{F} \rightarrow  \mathcal{F}_{t}$ with respect to the transversal family $\mathcal{N}$.  
   
   The leaf conjugacy $\mathfrak{h}_{1}$ sends $\mathcal{F} $-plaques to $\mathcal{F}_{1}$-plaques.  Since $h_{p}(x)$ and $h_{x}(x)$ lie in the same $\mathcal{F}_{1}$-plaque, and since $h_{x}(x) = \mathfrak{h}_{1}(x)$ we see that $h_{p}(x)$ is the intersection of $N(p,r)$ with the $\mathcal{F}_{1}$-plaque $\mathfrak{h}_{1}(\rho (x))$.  See Figure~\ref{f.coherence}. 
   \begin{figure}[htbp]
\centering
\includegraphics[scale=.60]{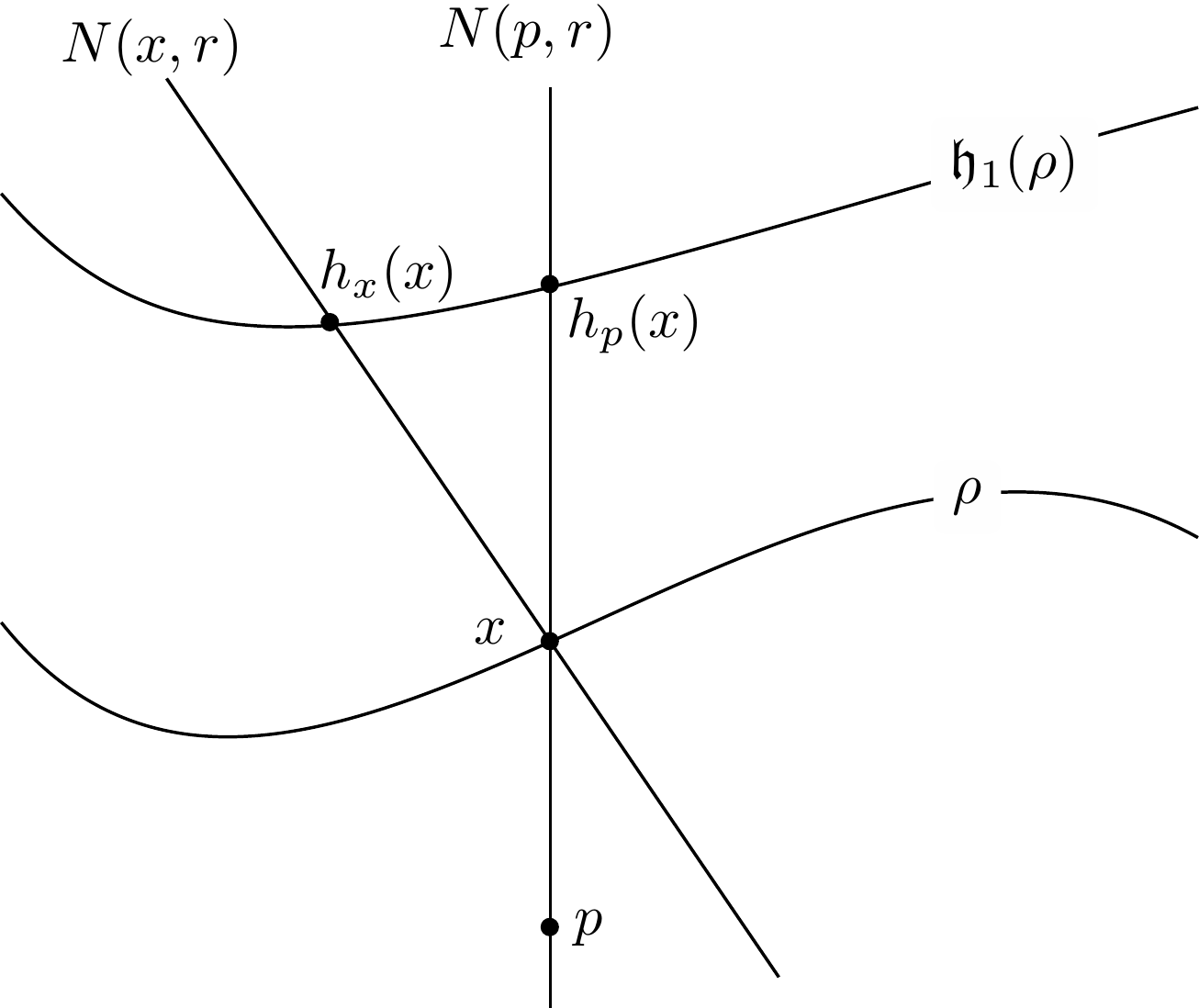}
\caption{Both holonomy maps $h_{p}$ and $h_{x}$ send $x$ into the same $\mathcal{F}_{1}$-plaque, namely the plaque $\mathfrak{h}_{1}(\rho )$ through $\mathfrak{h}_{1}(x) = h_{x}(x)$.}
\label{f.coherence}
\end{figure}

By hypothesis, the holonomy maps $h_{p}$ are uniformly $\theta $-biH\"older.  We    claim     there is a constant $H$ such that for all   nearby  $p, q \in M$ we have
$$
\frac{d(p,q)^{1/\theta }}{H} \leq d(\mathfrak{h}_{1}(p), \mathfrak{h}_{1}(q)) \leq  Hd(p,q)^{\theta } \ .
$$ 
Figure~\ref{f.coherence2} indicates two geodesic triangles.  They are effectively right triangles with hypotenuses shown as dotted lines.  For their angles at  $x$  and  $h_{p}(x)$ do not differ much from $\pi /2$. The ratio  of the larger leg to the hypotenuse is bounded between $1/K$ and $K$ where $K$ is a constant.  
\begin{figure}[htbp]
\centering
\includegraphics[scale=.60]{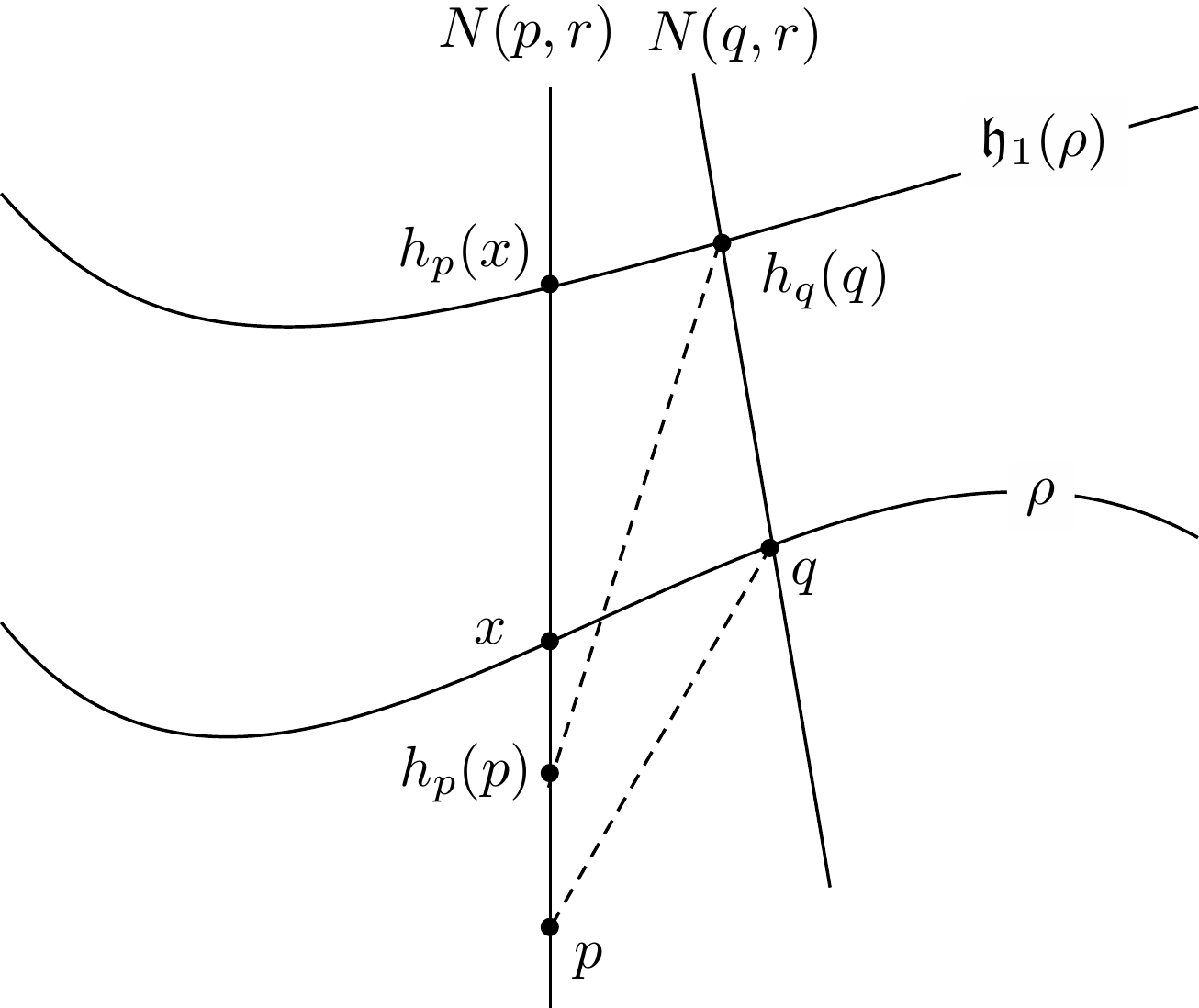}
\caption{$\rho $ is the plaque that contains $q$.  It meets $N(p, r)$ at $x$.  The triangles involved are essentially right triangles with the dotted lines as hypotenuses.}
\label{f.coherence2}
\end{figure}

Thus
\begin{equation*}
\begin{split}
d(\mathfrak{h}_{1}(p),\mathfrak{h}_{1}(q)) &= d(h_{p}(p), h_{q}(q)) 
\\
&\leq  K  \max\{  d(h_{p}(p), h_{p}(x)), d(h_{p}(x), h_{q}(q)) \}
\\
&\leq  K  \max\{  H_{0}d(p,x)^{\theta },   Dd(x,q) \}
\\
& \leq  K \max\{ H_{0},D\} \max\{ d(p,x), d(x,q)\}^{\theta }  
\\
&\leq K^{1+ \theta }\max\{ H_{0},D\}  d(p,q)^{\theta }  ,
\end{split}
\end{equation*}
where  $D$ bounds the distortion of the holonomy along $\mathcal{N}(\rho )$ from one plaque to another.  Similarly
\begin{equation*}
\begin{split}
Kd(h_{p}(p), h_{q}(q)) & \geq \max\{d(h_{p}(p), h_{p}(x)) , d(h_{p}(x), h_{q}(q))\}
\\
&\geq \max\{ \frac{d(p,x)^{1/\theta }}{H_{0}} ,  \frac{d(x,q)}{D} \}
\\
& \geq  \frac{\max \{ d(p,x) , d(x,q) \}^{1/\theta  } }{\max \{H_{0}, D\}} 
\\
& \geq \frac{1}{K^{1/\theta }\max \{H_{0}, D\}}d(p,q)^{1/\theta } \ ,
\end{split}
\end{equation*}
which completes the proof that $\mathfrak{h}_{1}$ is $\theta $-bih\"older with biH\"older constant   $H = K^{1+1/\theta }\max\{ H_{0}, D\}$.
 \end{proof}

\begin{Rmk}
We have used the perturbation theory of \cite{HPS77} for a diffeomorphism of a manifold with boundary, namely $T \times  M$ when $T = [0, 1]$.  To avoid waiving our hands and saying that the whole theory in \cite{HPS77} works also on manifolds with boundary, it is simpler to replace $[0,1]$ by the circle.  The   path of diffeomorphisms $g_{t}$ with $0 \leq  t \leq  1$ is replaced by a  $C^1$ loop of 
 diffeomorphisms, say $g_{t}$ with $0 \leq  t \leq  2$, all of which $C^{1}$-approximate $f $.  In this way $F $ and $G$ act on a compact manifold without boundary, namely  $S^{1} \times  M$, and we get to treat the leaf conjugacies and holonomy as we did when $T = [0,1]$.   
\end{Rmk}

\begin{Add}
\label{a.lamination}
The preceding theorem   holds also for laminations.
\end{Add}

\begin{proof}[\bf Proof]
The reasoning is exactly the same for laminations as for foliations.
\end{proof}

\section{A Uniform H\"older Section Theorem}
\label{s.C1PUHST}

Consider a fiber contraction
$$
\begin{CD}
 W @>\text{\normalsize
$\qquad 
F \qquad$}>> 
W
\\
@V\text{\normalsize$
\pi 
$}VV @VV\text{\normalsize$
\pi 
$}V
\\
X @>\text{\normalsize$\qquad 
h 
\qquad$}>>
X
\end{CD}
$$
It contracts the fibers  uniformly and has a unique invariant section $\sigma _{F}$.  Our goal  here  is a theorem asserting that $\sigma _{F}$ is $\theta $-H\"older when the fiber contraction dominates the base contraction at scale $\theta $.  Previous versions of such a result appear in \cite{HPS77}, Shub's book \cite{Shub}, and Wilkinson's paper \cite{W97} under hypotheses sometimes involving a compact base space, $C^{2}$ differentiability, and global bundle triviality, all of  which we need to relax.  Compactness becomes uniformity, $C^{2}$ becomes $C^{1}$, and global bundle triviality becomes local bundle triviality.

Existence, uniqueness, and continuity of the invariant section $\sigma _{F}$ of a fiber contraction are straightforward.  The standard assumptions are that 
\begin{itemize}

\item[(a)]
$F$ is continuous, $\pi $ is a continuous surjection, $h$ is a homeomorphism, and each fiber  $\pi ^{-1}(x)$ is equipped with a    metric  $d_{x}$ which makes the fiber complete,   depends continuously on $x \in X$, and is uniformly bounded.
\item[(b)]
There is a $k < 1$ such that 
$$
d_{h(x)}(F(w), F(w^{\prime})) \leq kd_{x}(w,w^{\prime})
$$
for all $w,w^{\prime} \in \pi ^{-1}(x)$ and all $x \in X$. 
\item[(c)]
There exists a continuous section $\sigma _{0} : X \rightarrow  W$. \end{itemize}
Then the space $\Sigma ^{c}$ of continuous sections is metrized by 
$$
d(\sigma ,\sigma ^{\prime}) = \sup_{x} d_{x}(\sigma (x), \sigma ^{\prime}(x))  
$$
and is complete.  It is contracted by the graph transform $F_{\#} : \Sigma ^{c} \rightarrow  \Sigma ^{c}$, 
$$
F_{\#} : \sigma \mapsto F \circ  \sigma  \circ  h^{-1} \ .
$$
The unique fixed point of $F_{\#}$ is the invariant section $\sigma _{F}$.  Under $F_{\#}$-iteration every section $\sigma \in \Sigma ^{c}$ converges uniformly to $\sigma _{F}$.  

To show that $\sigma _{F}$ is H\"older we want to justify the assertion that $F_{\#}$ leaves invariant a closed subspace of H\"older sections, and therefore $\sigma _{F}$ lies in that subspace.  

An initial H\"older assumption is that the fiber contraction $\theta $-dominates the base contraction: If $X$ is metrized this means that for all $x \in M$  
$$
k(x) < \mu (x)^{\theta }
$$
 where $k(x) \leq  k  < 1$ is the   Lipschitz constant of $F$ restricted to  the fiber $\pi ^{-1}(x)$ and $\mu (x)$ is the reciprocal of the Lipschitz constant of $h^{-1}$ at $h(x)$.  Without such a dominance    condition   H\"olderness can fail. 
 
  A second H\"older assumption concerns the vertical shear of $F$.   It  measures how much $F$ slides fibers up and down.  For example, take  
$$
F : (x,y) \mapsto (x/9,y/3 + \sin (50 \, x)) \ ,
$$
and see Figure~\ref{f.largeshear}.
\begin{figure}[htbp]
\centering
\includegraphics[width=300pt,height=100pt]{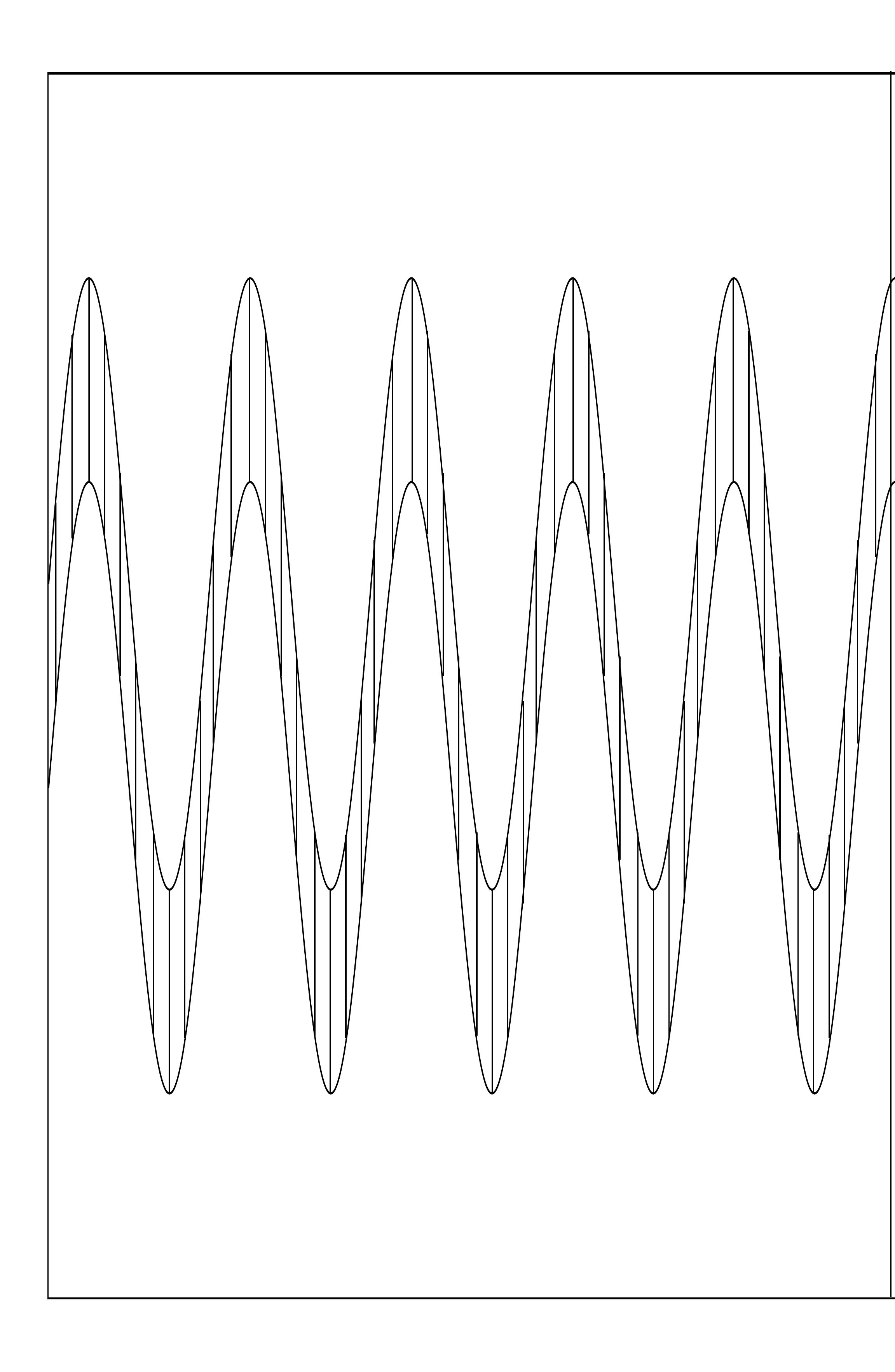}
\caption{The fiber contraction is $1/3$, the base contraction is $1/9$, and for all $\theta  <1/2$, the fiber contraction $\theta $-dominates the base contraction.  The vertical shear  is large but finite.}
\label{f.largeshear}
\end{figure}

  Bundle charts give the best way to quantify   vertical shear.   
 A bundle chart is a map $\varphi $ that sends $U\times Y$ homeomorphically  onto an open subset of $W$ such that each $\varphi (u \times  Y)$ is a fiber $  \pi ^{-1}(\xi (u))$.  This defines a base chart   $\xi : U \rightarrow  X$ so that  
$$
\begin{CD}
 U\times Y @>\text{\normalsize
$\qquad 
\varphi  \qquad$}>> 
W
\\
@V\text{\normalsize$
\pi 
$}VV @VV\text{\normalsize$
\pi 
$}V
\\
U @>\text{\normalsize$\qquad 
\xi 
\qquad$}>>
X
\end{CD}
$$
commutes.  We assume $U$, $Y$ are metric spaces, $\xi (U)$ is open in $X$, and $\xi : U \rightarrow  \xi (U)$ is a  homeomorphism.  

A bundle atlas for $W$ is a collection $\mathcal{A}$ of bundle charts $\varphi _{i}$ that cover $W$.  The corresponding collection $\mathcal{B}$ of base charts $\xi _{i}$ is an atlas for $X$.   
  We write $W_{i} = \varphi _{i}(U \times  Y)$, $X_{i} = \xi _{i}(U)$, and $d_{i}(x,x^{\prime}) = d_{U}(\xi _{i}^{-1}(x), \xi _{i}^{-1}(x^{\prime}))$.  The base atlas \textbf{covers $X$ uniformly} provided there is a $\delta  > 0$ such that if $x, x^{\prime} \in X_{j}$ and $d_{j}(x,x^{\prime}) < \delta $ then some $X_{i}$ contains the pair $h^{-1}(x), h^{-1}(x^{\prime})$.

  If $F(W_{i}) \cap W_{j} \not= \emptyset$ then 
$$
F_{ij}(u,y) = \varphi _{j}^{-1} \circ  F \circ \varphi _{i}(u,y) = (h_{ij}(u), v_{ij}(u,y)) \in U \times  Y \ ,
$$
is a chart expression for $F$.  The letters $h$ and $v$ indicate the horizontal and vertical components of $F$.    
The   \textbf{vertical shear} of $F_{ij}$  is $V_{ij}(u, u^{\prime}, y) = d_{Y}(v_{ij}(u,y), v_{ij}(u^{\prime},y)) $.  If there is a constant $L$ such that for all $ij$ and $u, u^{\prime} \in U$, $y \in Y$  we have
$$
V_{ij}(u,u^{\prime},y ) \leq  L d_{U}(u,u^{\prime})^{\theta }
$$
then $F$ has \textbf{$\pmb{\theta }$-bounded vertical shear} with respect to $\mathcal{A}$.  It is natural to define the chartwise fiber contraction and base contraction as $k_{ij}$ and $\mu _{ij}$ where $k_{ij}$ is the supremum of the  Lipschitz constants of $y \mapsto v_{ij}(u,y)$ with $u \in U$,  and $\mu _{ij}$ is the reciprocal of the Lipschitz constant of $  h_{ij}^{-1}$.
\begin{Defn}
The fiber contraction $F$ is \textbf{uniformly $\pmb{\theta }$-H\"older} with respect to $\mathcal{A}$ if the base atlas covers $X$ uniformly,   $F$  has uniformly $\theta $-bounded vertical shear, and the fiber contraction uniformly  $\theta $-dominates the base contraction in the sense that  
$$
\sup_{ij} \frac{k_{ij}}{\mu _{ij}^{\theta }} < 1 \ .
$$
\end{Defn}

The \textbf{principal part} of a section $\sigma : X \rightarrow  W$ in the bundle chart $\varphi _{i}$ is the map $s_{i} :  X_{i} \rightarrow Y$ such that 
$$
\varphi _{i}^{-1} \circ  \sigma (x) = (\xi _{i}^{-1}(x), s_{i}(x)) \ .
$$
Its   $\theta $-H\"older constant is 
 $$
H(s_{i}) =  \sup \frac{d_{Y}(s_{i}(x),s_{i}(x^{\prime}))}{d_{i}(x,x^{\prime})^{\theta }}
 $$ 
 where the supremum is taken over all $x, x^{\prime} \in X_{i}$ with $x \not=  x^{\prime}$.  If $\sup_{i} H(s_{i}) < \infty$ then the section is \textbf{uniformly $\pmb{\theta }$-H\"older} with respect to $\mathcal{A}$.

 \begin{Thm} 
  \label{t.UHST}
 With respect to $\mathcal{A}$, if
  $F$ is a uniformly $\theta $-H\"older  fiber contraction   and   there exist uniformly $\theta $-H\"older sections then  
$\sigma _{F}$ is uniformly  $\theta $-H\"older. 
 \end{Thm}

 \begin{Rmk}
We refer to Theorem~\ref{t.UHST} as a pointwise (or relative) result in contrast to an absolute result  because  we are comparing fiber contraction to base contraction over small neighborhoods rather than comparing the weakest fiber contraction over all of $W$ to the sharpest base contraction over all of $X$.   
\end{Rmk}

\begin{Rmk}
  The space $X$ is built from homeomorphic copies of $U$ but it may be non-separable.  In our application    $\mathcal{A}$ is locally finite but     uncountable. 
\end{Rmk}

\begin{proof}[\bf Proof]  Let $k_{ij}$ and $\mu _{ij}$ denote the fiber contraction and base contraction of $F_{ij} = (h_{ij}(u), v_{ij}(u,y))$,
\begin{equation*}
\begin{split}
k_{ij} &= \sup_{y\not= y^{\prime}} \frac{d_{Y}(v_{ij}(u,y), v_{ij}(u,y^{\prime}))}{d_{Y}(y,y^{\prime})}
\\
\mu _{ij} &= \inf _{u\not= u^{\prime}}\frac{d_{U}(h_{ij}(u), h_{ij}(u^{\prime}))}{d_{U}(u,u^{\prime})}  \ .
\end{split}
\end{equation*}

By assumption there exists a  section $\sigma _{0}$ whose principal parts with respect to $\mathcal{A}$ are  uniformly $\theta $-H\"older.  Let $H_{0}$ be a bound for these H\"older constants.     Choose 
$$
H = \max\{ H_{0}, D/\delta ^{\theta } ,  \sup _{ij}1/(\mu_{ij} ^{\theta} -k_{ij})\}
$$
where $D$ is the diameter of $Y$ and $\delta $ is the base covering constant.  (If $x, x^{\prime} \in X_{j}$ and $d_{j}(x,x^{\prime}) \leq  \delta $ then  $h^{-1}(x), h^{-1}(x^{\prime})$ lie in a common $X_{i}$.  Since $\sup_{ij} k_{ij}/\mu _{ij} ^{\theta } < 1$, $H$ is finite.)   
Let   $\Sigma ^{\theta ,H}$ be the set of sections whose principal parts have $\theta $-H\"older constant $\leq  H$.   They satisfy 
$$
d_{Y}(s_{j}(x), s_{j}(x^{\prime}))\leq H d_{j}(x,x^{\prime})^{\theta }  
$$
for all $x, x^{\prime} \in X_{j}$ and all $j$.  The set  $\Sigma ^{\theta ,H}$ is  nonempty since it contains $\sigma _{0}$.  It is closed since   if a   sequence of sections   converges with respect to the metric $d$ on $\Sigma $ then in any chart its principal parts converge uniformly, and   H\"older conditions survive uniform convergence.  

It remains to show that $\Sigma ^{\theta ,H}$ is $F_{\#}$-invariant.  Given $\sigma \in \Sigma ^{\theta ,H}$ and a chart $\varphi _{j}$ we must show that the principal part of $\widetilde{\sigma } = F_{\#}(\sigma )$ satisfies
$$
d_{Y}(\widetilde{s}_{j}(x), \widetilde{s}_{j}(x^{\prime}))\leq H d_{j}(x,x^{\prime})^{\theta } \ .
$$

\underline{Case 1.}  $d_{j}(x,x^{\prime}) > \delta   $.    Then
$$
d_{Y}(\widetilde{s}_{j}(x), \widetilde{s}_{j}(x^{\prime}))\leq D = (D/\delta  ^{\theta } )\, \delta   ^{\theta } \leq H d_{j}(x,x^{\prime})^{\theta } 
$$
since $D/\delta  ^{\theta } \leq  H$.

\underline{Case 2.}  $d_{j}(x,x^{\prime}) \leq  \delta   $.  Then   there is an   $X_{i}$   containing the pair $h^{-1}(x), h^{-1}(x^{\prime})$.  The principal parts of $\sigma _{i}$ and $\widetilde{\sigma} _{j}$ are related by the formula   
$$
\widetilde{s}_{j} (x) = v_{ij}(u,s_{i}(h^{-1}(x)))
$$
where    $\xi _{i}(u) = h^{-1}(x)$.  This is proved by formula   manipulation.  We have   $F_{ij} = \varphi _{j}^{-1}\circ F\circ \varphi _{i}$, and thus $\varphi _{j}^{-1}\circ F =   F_{ij} \circ  \varphi _{i}^{-1}$, which gives
\begin{equation*}
\begin{split}
\varphi _{j}^{-1}\circ \widetilde{\sigma }(x) & = \varphi _{j}^{-1}\circ F\circ \sigma \circ h^{-1}(x)
\\
&= F_{ij} \circ \varphi _{i}^{-1}\circ \sigma \circ h^{-1}
\\
&= F_{ij}(u,s_{i}(h^{-1}(x))
\\
&= (h_{ij}(u), v_{ij}(u, s_{i}(h^{-1}(x))) \ .
\end{split}
\end{equation*}
 Since $\varphi _{j}^{-1}\circ \widetilde{\sigma }(x) = (\xi _{j}^{-1}(x), \widetilde{s}_{j}(x))$, we equate the vertical components to get the relation between the principal parts of $\sigma _{i}$ and  $\widetilde{\sigma }_{j}$ as stated. 
 
Since $\sigma \in \Sigma ^{\theta ,H}$ we have 
$$
d_{Y}(s_{i}(h^{-1}(x)), s_{i}(h^{-1}(x^{\prime}))\leq H d_{i}(h^{-1}(x), h^{-1}(x^{\prime}))^{\theta } \ .
$$
Thus 
  \begin{equation*}
\begin{split}
d_{Y}(\widetilde{s}_{j}(x),\widetilde{s}_{j}(x^{\prime})) &= d_{Y}(v_{ij}(u, s_{i}(h^{-1}(x)), v_{ij}(u^{\prime}, s_{i}(h^{-1}(x^{\prime}))))
\\
& \leq d_{Y}(v_{ij}(u, s_{i}(h^{-1}(x)), v_{ij}(u , s_{i}(h^{-1}(x^{\prime} ))))
\\
&+
d_{Y}(v_{ij}(u , s_{i}(h^{-1}(x^{\prime})), v_{ij}(u^{\prime}, s_{i}(h^{-1}(x^{\prime}))))
\\
& \leq    k_{ij}d_{Y}(s_{i}(h^{-1}(x), s_{i}(h^{-1}(x^{\prime})) + Ld_{U}(u,u^{\prime})^{\theta } 
\\
&\leq   k_{ij}Hd_{i}(h^{-1}(x), h^{-1}(x^{\prime})))^{\theta } + Ld_{i}(h^{-1}(x), h^{-1}(x^{\prime}))^{\theta } 
\\
&\leq (( k_{ij}H + L)/\mu_{ij} ^{\theta })d_{j}(x,x^{\prime})^{\theta } \leq Hd_{j}(x,x^{\prime})^{\theta } 
\end{split}
\end{equation*}
where $\xi _{i}(u) = h^{-1}(x)$ and $\xi _{i}(u^{\prime}) = h^{-1}(x^{\prime})$.  This completes the proof that $\Sigma ^{\theta ,H}$ is $F_{\#}$-invariant and therefore that $\sigma _{F}$ is $\theta $-H\"older.
\end{proof}

We will use the the Uniform H\"older Section Theorem as follows.  The base manifold will be the disjoint union of the global strong stable manifolds of a partially hyperbolic diffeomorphism $g$.  The fiber at $x \in W^{s}(p)$ is the local strong stable manifold at $x$, $W^{s}(x,r) $.  (Although this makes the fiber a subset of the base we can still think of a bundle this way.)  The bundle  map $F$ approximates the product $g \times  g$.  The base map   is not $g$ but is an amalgam $a$ of $g$ and a nearby diffeomorphism $f$.  Both fiber and base are contracted, but  for some $\theta $,  the fiber contraction dominates the $\theta ^{\textrm{th}}$ power of the base contraction.  

We cover the manifold $M$ with finitely many charts inside which the bundle is trivial.  This gives uncountably many charts $\varphi _{i}$ that cover the bundle because the base manifold $V$ has uncountably many components.  However the chart expressions for $F$ are uniform, and uniformity overcomes non-compactness.

\section{The Proof of Theorem~A}
\label{s.A}

We are given a normally hyperbolic diffeomorphism $f$  whose center unstable, center, and center stable foliations are $C^{1}$.  We intend to show that the corresponding foliations of a $C^{1}$ perturbation $g$ of $f$ are H\"older.  Most of the work is contained in the following result, which will be applied to the suspension of $\mathcal{W}^{cu}  $ and the suspension of $g$.   

Let $\mathcal{E}$ be a $C^{1}$ foliation at which $f$ is normally contracting, say  
$$
0 < \mu  <T^{s}f < \nu   < T_{T\mathcal{E}}f   \quad \textrm{and} \quad \nu < \mu ^{\theta } < 1 \ ,
$$
as in the normally hyperbolic case.   (The letter $\mathcal{E}$ is meant to suggest a     foliation along whose leaves $f$ is at least somewhat expanding.  An example is $\mathcal{E} = \mathcal{W}^{cu}$.)  
  Let $g$ $C^{1}$-approximate $f$ and let $\mathcal{W}^{s}_{g}$ be its strong stable foliation.  Being $C^{1}$,   $f$ is plaque expansive with respect to $\mathcal{E}$, and so   there is a unique nearby $g$-invariant foliation $\mathcal{E}_{g}$ and there is a canonical leaf conjugacy $\mathfrak{h} : \mathcal{E} \rightarrow  \mathcal{E}_{g}$ that respects $\mathcal{W}^{s}_{g}$.  
 
\begin{Propn}
\label{p.A}
$\mathfrak{h}$ is uniformly $\theta $-biH\"older when restricted to the $\mathcal{W}^{s}_{g}$-leaves  and the $\mathcal{E}_{g}$-holonomy is $\theta ^{2}$-H\"older.   
\end{Propn}

\begin{proof}[\bf Proof]   
We define a map $a : M \rightarrow  M$ which is an ``amalgam'' of $f$, $\mathcal{E} $,    $g$, and $\mathcal{W}^{s}_{g}$ as follows.  For each $x \in M$ we set
  $$
  a(x) =W^{s}_{g}(g(x),r) \cap \mathcal{E} (f(x),r) \ .
$$
Here $r > 0$ is small and fixed.  $r$ is a radius where the local invariant manifolds of $f$ are fairly flat, meaning that distance measured along them approximates geodesic distance in $M$.  If the sup-distance between $f$ and $g$ is $< r/2$ then   $a(x)$ is uniquely defined by transversality.  To get the necessary estimates on $a$ we assume
$$
d_{C^{1}}(f,g) \ll r \ .
$$
 The way to think of $a$ in terms of $\mathcal{W}^{s}_{g}$ is this.  It takes the   stable $g$-manifold at $x$ and sends it to the stable $g$-manifold at $g(x)$, but it does not do so simply by applying $g$ and it does not send the base point $x$ to  the base point $g(x)$.  Rather, it first applies $f$  and then slides along the   $\mathcal{E} $-foliation to get to $W^{s}_{g}$.  Since everything is local, the sliding distance is small and we see that
 $$
 \textrm{\emph{$a$-orbits are $f$ pseudo-orbits.}}
 $$

  Since   $\mathcal{E} $ is $C^{1}$,  the amalgam map $a : M \rightarrow  M$ is   $C^{1}$ along the   $\mathcal{W}^{s}_{g}$-leaves.  The $C^{1}$ hypothesis on $\mathcal{E}$ is crucial here.  Restricted to the $\mathcal{W}^{s}_{g}$-leaves the amalgam map    $C^{1}$-approximates $g$ which $C^{1}$-approximates $f$.  Therefore, like $g$, the amalgam map    contracts  the   stable $g$-manifolds by a factor   $< \nu $.   See Figure~\ref{f.fga}.    
\begin{figure}[htbp]
\centering
\includegraphics[scale=.60]{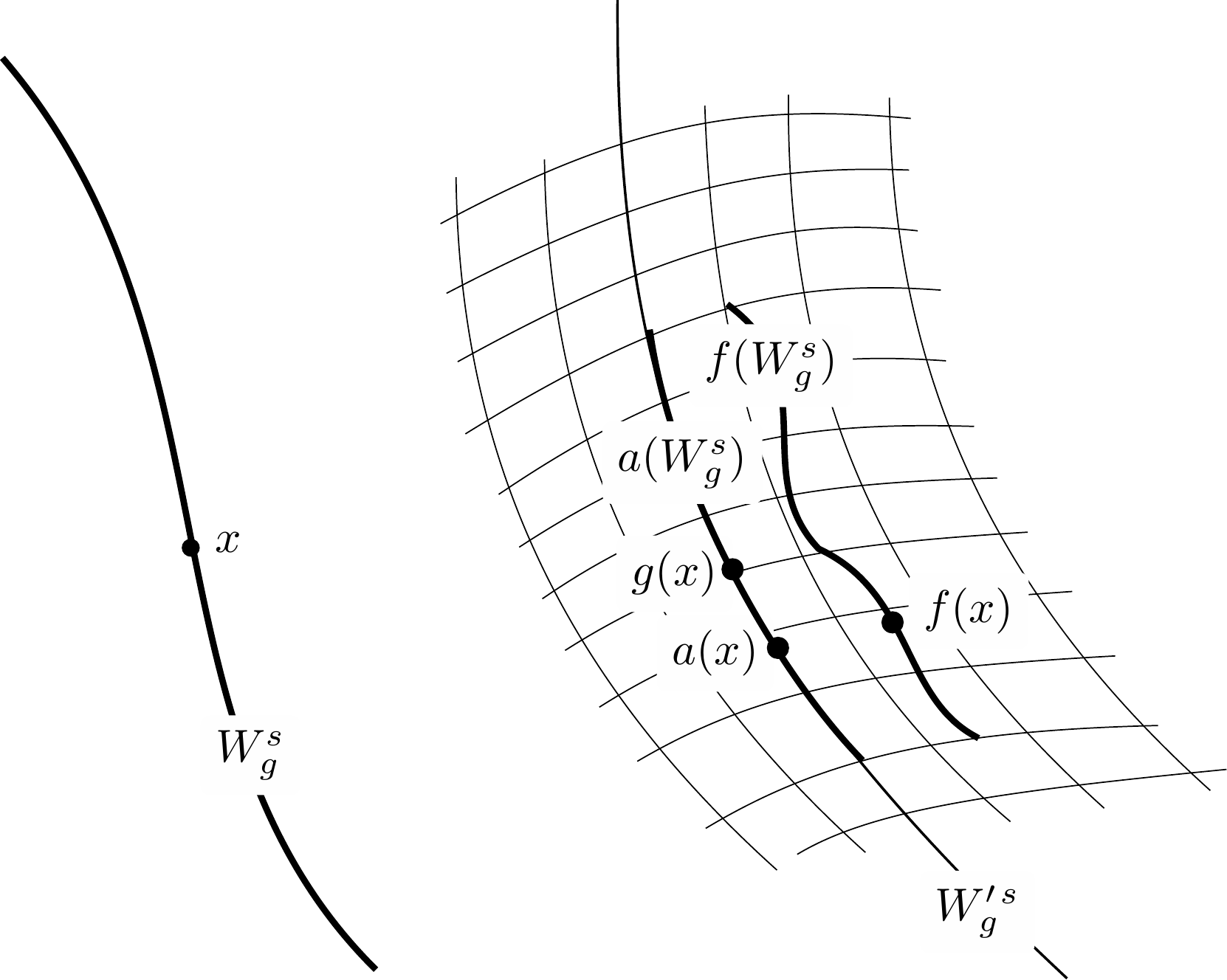}
\caption{The amalgam map $a$ and its effect on a local stable $g$-manifold $W^{s}_{g}$.  The light horizontal curves are the $\mathcal{E} $-leaves and the light vertical curves are the $\mathcal{W}^{s}_{g}$-leaves.  The   curves labelled $W^{s}_{g}$ and $W^{\prime}_{g}{}^{s}$ represent    $\mathcal{W}^{s}_{g}$-leaves of radius $r$ based at $x$ and $g(x)$ respectively.  Note that $f$ shrinks $W^{s}_{g}$ by a factor $<\nu $ since $W^{s}_{g}$ is approximately tangent to $E^{s}$.}
\label{f.fga}
\end{figure}

We next define two nonlinear fiber contractions and study their invariant sections.  In both cases the total space is 
$$
W = \{(x,y) \in M \times  M : y \in W^{s}_{g}(x,r)\} 
$$
and the projection     $\pi : W \rightarrow  M$ is $\pi (x,y) = x$.  \emph{A priori the bundle $(W,\pi , M)$ is not $C^{1}$.}  For although $M$ and $\pi $ are   smooth, the set $W$ is locally the graph of $\mathcal{W}^{s}_{g}$, a foliation quite likely to be H\"older at best.

The fiber maps are products, $F(x,y) = (a(x),g(y))$ and $G(x,y) = (g(x),a(y))$,
$$
\begin{CD}
 W @>\text{\normalsize
$\qquad 
F \qquad$}>> 
W
\\
@V\text{\normalsize$
\pi  
$}VV @VV\text{\normalsize$
\pi  
$}V
\\
M @>\text{\normalsize$\qquad 
a 
\qquad$}>>
M
\end{CD}\qquad \qquad 
\begin{CD}
 W @>\text{\normalsize
$\qquad 
G \qquad$}>> 
W
\\
@V\text{\normalsize$
\pi 
$}VV @VV\text{\normalsize$
\pi 
$}V
\\
M @>\text{\normalsize$\qquad 
g 
\qquad$}>>
M 
\end{CD} 
$$
The $\pi $-fiber at $x$ is $W^{s}_{g}(x,r)$, and it is contracted by $g$ into $W^{s}_{g}(g(x),\nu r)$.  The $\pi $-fiber over $a(x)$ is $W^{s}_{g}(a(x), r)$, and since  $a(x) \in W^{s}_{g}(g(x), r)$ is much closer to $g(x)$ than $\nu r$,   we have
 $$
  g(W^{s}_{g}(x,r)) \subset  W^{s}_{g}(g(x),\nu r)  \subset 
 W^{s}_{g}(a(x), r) \ ,
 $$
which means that $F$ contracts fibers over $a$.     For the same reasons we have 
$$
a(W_{g}^{s}(x,r)) \subset W^{s}_{g}(a(x), \nu r) \subset W^{s}_{g}(g(x), r) \ ,
$$
which means that $G$ contracts fibers over $g$.
See Figure~\ref{f.fga}.

 Fiber contractions have unique invariant sections.
  Let $\sigma : M \rightarrow W$  be the $F$-invariant section and $\tau : M \rightarrow W$ be the $G$-invariant section.   Since $W \subset  M \times  M$ we can write
  $$
  \sigma (x) = (x, s(x)) \qquad \tau (x) = (x, t(x)) \ .
  $$
The maps $s, t : M \rightarrow  M$ are the \textbf{principal parts} of the sections $\sigma , \tau $.    We claim that $s$ and $t$ are inverse leaf conjugacies between $\mathcal{E} $ and $\mathcal{E}_g $.  
  
  Since $\sigma (x) = (x, s(x))$ is $F$-invariant and $F = a \times g$, the $a$-orbit of $x$ and the $g$-orbit of $s(x)$ shadow each other closely.  Since $a$-orbits are $f$ pseudo-orbits, the $g$-orbit of $s(x)$ is closely shadowed by an $f$ pseudo-orbit through $x$.  Therefore     $s : \mathcal{E} \rightarrow  \mathcal{E}_g$ is a canonical leaf conjugacy.  
It has an inverse map, $s^{-1}$.  Now, 
    $F$-invariance implies that $g(s(x)) = s(a(x))$, and therefore 
 $
  s^{-1}((g(x)) = a(s^{-1}(x)) 
 $.
  Consequently 
  $$
  G(x,s^{-1}(x)) = (g(x), a(s^{-1}(x)) = (g(x), s^{-1}(g(x))
  $$
   which implies that $x \mapsto (x, s^{-1}(x))$ is a $G$-invariant section, i.e., it is the unique $G$-invariant section $\tau $.  Equal sections have equal principal parts, so  $s^{-1} = t$.

 The amalgam mapping   $a : M \rightarrow  M$  is a homeomorphism which carries  $W^{s}_{g}(x,r)$ into $W^{s}_{g}(g(x), r)$ $C^{1}$ diffeomorphically.   Therefore it sends  the global stable manifold $W^{s}_{g}(x)$ diffeomorphically  onto the global stable manifold $W^{s}_{g}(g(x))$.  When we put the leaf topology on $M$ with respect to the foliation $\mathcal{W}^{s}_{g}$ we get a $C^{1}$ non-separable manifold $V$ of dimension $s$ whose connected components are the global stable manifolds $W^{s}_{g}(x)$, and $a : V \rightarrow  V$ is a $C^{1}$ diffeomorphism.  
 
 We have a bundle $\mathcal{W}$ over $V$.  Its total space is $W$ with the leaf topology on its base $V = \bigcup W^{s}_{g}$.  Its  fiber at $x$ is $W^{s}_{g}(x,r)$.    On $\mathcal{W}$ we have two fiber contractions, $F$ and $G$. We claim that $\mathcal{W}$   is uniformly $C^{1}$ and  the Uniform H\"older Section Theorem (Theorem~\ref{t.UHST}) applies to it.  
 
 The fiber  $W^{s}_{g}(x,r)$ varies in a fairly trivial $C^{1}$ fashion   as $x$ varies in the global stable manifold $W^{s}_{g}(x)$.  Since each connected component   $W^{s}_{g} \subset  V$ is simply connected, any disc bundle over it (such as $\mathcal{W}|_{W^{s}_{g}}$)  is trivial.  A priori this triviality is not uniform.  That is, there need be no relation between the trivializing vector fields at points of $W^{s}_{g}$ which are nearby each other in $M$ but distant along $W^{s}_{g}$.  
 This is why we need an invariant section theorem in which the bundle is only locally trivial.  
 
 The proof that $\mathcal{W}$ does have such a uniform $C^{1}$ bundle structure to which Theorem~\ref{t.UHST} applies is distractingly technical and relegated to the appendix following this section.  Admitting this, we see that the leaf conjugacy $s : \mathcal{E} \rightarrow  \mathcal{E}_g$ is uniformly $\theta $-H\"older, and so is the inverse  leaf conjugacy $t : \mathcal{E}_g \rightarrow  \mathcal{E}$, when   
$ \nu  < \mu ^{\theta }$.  

To complete the proof of Proposition~\ref{p.A}  we show that the $\mathcal{E}_{g}$-holonomy is $\theta ^{2}$-H\"older.

 Consider points $p, p^{\prime}$ with $p^{\prime} \in \mathcal{E}(p)$  and  draw a path $ \alpha \subset  \mathcal{E}(p)$ from $p$ to $p^{\prime}$.  Corresponding to $\alpha$ we have the $\mathcal{E}_{g}$-holonomy map $h_{g} $.  It sends $x \in W^{s}_{g}(p,r)$ to $x^{\prime} \in W^{s}_{g}(p^{\prime},r^{\prime})$ such that $x, x^{\prime}$ lie on a common $\mathcal{E}_{g}$-leaf $L_{g}$ and are joined by a path on $L_{g}$ that approximates $\alpha$.  The radii $r \leq r^{\prime}$ are small.  Thus, $h_{g} = s^{\prime} \circ  h \circ  t$  where $t$ is the restriction of the inverse leaf conjugacy $\mathcal{E}_{g} \rightarrow  \mathcal{E}$ to $W^{s}_{g}(p)$, $h$ is the $C^{1}$ holonomy map for $\mathcal{E}$,  and $s^{\prime}$ is the restriction of the forward leaf conjugacy $s : \mathcal{E} \rightarrow  \mathcal{E}_{g}$ to $W^{s}(p^{\prime},r^{\prime})$. See Figure~\ref{f.holonomy5}.
\begin{figure}[htbp]
\centering
\includegraphics[scale=.60]{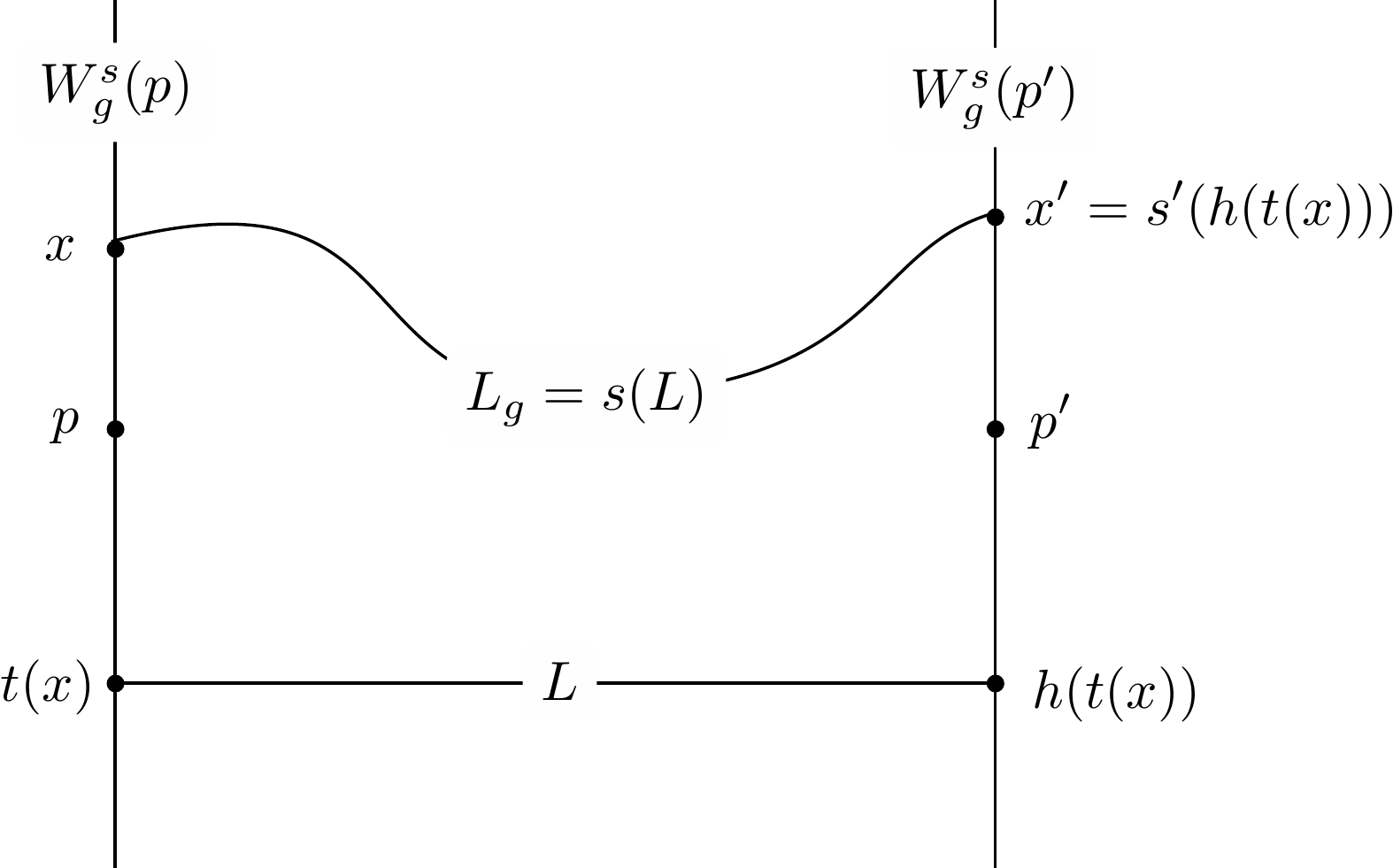}
\caption{The leaf conjugacy $s : \mathcal{E} \rightarrow  \mathcal{E}_{g}$ carries the $\mathcal{E}$-leaf $L$ through $t(x)$ to the $\mathcal{E}_{g}$-leaf  $L_{g}$ through $x$. The map $x \mapsto  x^{\prime}$ is the $\mathcal{E}_{g}$-holonomy.}
\label{f.holonomy5}
\end{figure}
 Since H\"older exponents multiply under composition,  the $\mathcal{E}_{g}$ holonomy maps are $\theta ^{2}$-H\"older.
\end{proof}

\begin{proof}[\bf Proof of Theorem~A]  We are given a diffeomorphism $f$ that is normally hyperbolic at a foliation $\mathcal{F}$.  It is assumed that its splitting is $C^{1}$, so in particular   the summands $E^{cu}$, $E^{c} = T\mathcal{F}$, and $E^{cs}$ integrate to unique $C^{1}$ $f$-invariant foliations $\mathcal{W}^{cu}$, $\mathcal{W}^{ c}= \mathcal{F}$, and $\mathcal{W}^{cs}$.  The diffeomorphism $f$ is normally hyperbolic at $\mathcal{W}^{cu}$ with respect to the splitting $T\mathcal{W}^{cu} \oplus E^{s}$.  Likewise, $f$ is normally hyperbolic at $\mathcal{W}^{c} = \mathcal{F}$ and at $\mathcal{W}^{cs}$.  Since the foliations are $C^{1}$ they are plaque expansive and structurally stable:   If $g$ $C^{1}$-approximates $f$ then  we have unique nearby $g$-invariant foliations $\mathcal{W}^{cu}_{g}$, $\mathcal{W}^{c}_{g} = \mathcal{F}_{g}$, $\mathcal{W}^{cs}_{g}$, and we have canonical leaf conjugacies from the $f$-invariant foliations to the corresponding $g$-invariant foliations.  

We want to show that   these leaf conjugacies   are biH\"older, and the   holonomy maps along the $g$-invariant foliations are   H\"older.  

 Let $T$ be the circle of length $2$.  The $f$-invariant foliation  $\mathcal{W}^{cu}$ is  $C^{1}$ and normally contracting, so the product   foliation
 $$
\mathcal{S}^{cu}  = T \times  \mathcal{W}^{cu}  \subset  T \times  M
 $$
   is also $C^{1}$ and  normally contracting with respect to the diffeomorphism
 $$
 F (t, x) = (t, f(x)) \ .
 $$
Its leaves are products $T \times W^{cu}(x)$.   The contraction rates are the same for $F $ and  $f$. 
Since $g$ $C^{1}$-approximates $f$ there is a $C^{1}$ loop of diffeomorphisms $g_{t} : M \rightarrow  M$ that $C^{1}$-approximate $f$ such that $g_{0} = f = g_{2}$ and $g_{1} = g$.  
The diffeomorphism 
$$
G(t,x) = (t, g_{t}(x))
$$
$C^{1}$-approximates $F$.  (These suspension diffeomorphisms $F$ and $G$   are different from the fiber maps $F$ and $G$ in Proposition~\ref{p.A}.)     Proposition~\ref{p.A} applies to  the suspension diffeomorphisms  so there     exists a unique $G$-invariant foliation $\mathcal{S}^{cu}_{G}$ near $\mathcal{S}^{cu}$ all of whose holonomy maps are $\theta ^{2}$-H\"older when $\nu  < \mu ^{\theta }$. Since $\mathcal{W}^{cu}_{g}$ is a slice of $\mathcal{S}^{cu}_{G}$,   the $\mathcal{W}^{cu}_{g}$-holonomy maps are also $\theta ^{2}$-H\"older.   By Theorem~\ref{t.Amie1} the suspension holonomy of $\mathcal{S}^{cu}_{G}$ gives a canonical leaf conjugacy   $\mathfrak{h}^{cu} : \mathcal{W}^{cu} \rightarrow \mathcal{W}^{cu}_{g}$, which is  $\theta ^{2}$-biH\"older.

Similarly we have $\mathcal{S}^{cs} = T \times  \mathcal{W}^{cs}$ and  $ \mathcal{S}^{cs}_{G}$  whose holonomy maps are $ \theta ^{2}$-H\"older when $\widehat{\nu } < \widehat{\mu }^ \theta $. Since $\mathcal{W}^{cs}_{g}$ is a slice of $\mathcal{S}^{cs}_{G}$,   the $\mathcal{W}^{cs}_{g}$-holonomy maps are also $ {\theta} ^{2}$-H\"older.
By Theorem~\ref{t.Amie1} the suspension holonomy of $\mathcal{S}^{cs}_{G}$ gives a canonical leaf conjugacy   $\mathfrak{h}^{cs} : \mathcal{W}^{cs} \rightarrow \mathcal{W}^{cs}_{g}$, which is  $ {\theta} ^{2}$-biH\"older.

To complete the proof we take intersections.  
Theorem~\ref{t.intersection} implies that the intersection foliation $\mathcal{S}^{c}_{G} = \mathcal{S}(\mathcal{F}_{g})$ has $\theta ^{2}$-H\"older  holonomy. 
Since $\mathcal{W}^{c}_{g} = \mathcal{F}_{g}$ is a slice of $\mathcal{S}^{c}_{G}$,   the $\mathcal{W}^{c}_{g}$-holonomy maps are also  $\theta ^{2}$-H\"older. 
By Theorem~\ref{t.Amie1} the suspension holonomy of $\mathcal{S}^{c}_{G}$ gives a canonical $\theta ^{2}$-H\"older leaf conjugacy   $\mathfrak{h}  : \mathcal{F}  \rightarrow \mathcal{F} _{g}$.
\end{proof}
 
\begin{Rmk}
Why do we need suspension in the proof of Theorem~A from Proposition~\ref{p.A}?  After all,   Proposition~\ref{p.A} applies directly to $\mathcal{W}^{cu}$ and $\mathcal{W}^{cs}$ when they are $C^{1}$ and implies that the leaf conjugacies $\mathfrak{h}^{cu} : \mathcal{W}^{cu} \rightarrow  \mathcal{W}^{cu}_{g}$ and $\mathfrak{h}^{cs} : \mathcal{W}^{cs} \rightarrow  \mathcal{W}^{cs}_{g}$ are H\"older.  Unfortunately we have no conjugacy intersection result to conclude from this that the center leaf conjugacy $\mathfrak{h}^{c} : \mathcal{W}^{c} \rightarrow  \mathcal{W}^{c}_{g}$ is H\"older.  But we do have a holonomy intersection  result, namely Theorem~\ref{t.intersection}.  That is why we convert the leaf conjugacy question to a suspension holonomy question.  See also Remark~\ref{k.translation} in Section~\ref{s.cautionary}.  
\end{Rmk}

\section*{ Appendix.  Uniform Bundle Charts}

In the proof of Proposition~\ref{p.A} the bundle $\mathcal{W}$ has total space $W$, base space $V$, and dimension $2s$.  We will construct bundle charts $\varphi  : \mathbb{R}^{s}(r) \times  \mathbb{R}^{s}(r) \rightarrow W$ which are uniformly  $C^{1}$ and in which the fiber maps  $F, G$ are uniformly $C^{1}$.  ($F = a \times  g$ and $G = g \times  a$ are the fiber maps from Proposition~\ref{p.A}.  They are different from the diffeomorphisms $F$, $G$ in the proof of Theorem~A.)   By this we mean two things:
\begin{itemize}

\item 
The chart transfer maps for overlapping charts are $C^{1}$ and their first derivatives are uniformly bounded and uniformly continuous.
\item  
The chart expressions for $F$ and $G$   have uniformly bounded, uniformly continuous first derivatives.
 \end{itemize} 
   The metric in which we measure everything will be a smooth metric on $TM$.  In addition, all these derivatives and charts will be continuous from one component of $V$ to another.

We start by smoothing the splitting.  The original normally  hyperbolic splitting $TM = E^{u} \oplus E^{c} \oplus E^{s}$ is orthogonal with respect to a continuous adapted Riemann structure.  (The bundle $E^{c}$ is $T\mathcal{F}$.) The smoothed splitting
$$
TM = \widetilde{E}^{u} \oplus \widetilde{E}^{c} \oplus \widetilde{E}^{s}
$$
uniformly approximates the original splitting and is orthogonal with respect to a smooth Riemann structure that approximates the original Riemann structure.  This is standard.   Although the smoothed splitting is not invariant we have
$$ Tf = 
\begin{bmatrix}
   \widetilde{T}^{u}f    &    {\displaystyle *}   &   {\displaystyle *} 
\\
   {\displaystyle *}    &      \widetilde{T}^{c}f   &    {\displaystyle *}
  \\
   {\displaystyle *}    &   {\displaystyle *}   &       \widetilde{T}^{s}f
\end{bmatrix}
\quad \textrm{with respect to} \quad \widetilde{E}^{u} \oplus \widetilde{E}^{c} \oplus \widetilde{E}^{s}  \ ,
$$ 
where the off-diagonal terms are small and $\widetilde{T}^{u}f \oplus \widetilde{T}^{c}f \oplus \widetilde{T}^{s}f$ approximates $T^{u}f \oplus T^{c}f \oplus T^{s}f$.

Let $\rho $ be a plaque of   $\mathcal{W}^{s}$ at $p$.  If $\exp_{p}$ is the smooth exponential at $p$ then $\overline{\rho } = \exp_{p}^{-1}(\rho )$ is the graph of a $C^{1}$ function   $\gamma  : \widetilde{E}^{s}(p,r) \rightarrow \widetilde{E}^{cu}(p)$.  The function $\gamma $ has Lipschitz constant $\leq  1$, it has value $0$ at the origin of $\widetilde{E}^{s}(p)$,  and its derivative with respect to $x \in \widetilde{E}^{s}(p, r)$ is uniformly continuous with respect to $x, p$.    This is all because $\mathcal{W}^{s}$ is a regular foliation, the smoothed splitting approximates the original splitting,  and $r$ is small.  Let $\operatorname{gr}(\gamma ) : \widetilde{E}^{s}(p,r) \rightarrow T_{p}M$ denote the function whose image is the graph of $\gamma $. 

Now we define  charts on $\mathcal{W}$ in which to express this.  We      cover $M$ with finitely many smooth charts $  \mathbb{R}^{m} \rightarrow  U \subset  M$ over which the smoothed splitting is trivial.  We    choose smooth orthonormal trivializing vector fields  $Y_{1}, \dots  , Y_{s}$ for $\widetilde{E}^{s}_{U}$.  Then for $t_{1}, \dots  , t_{s}$ we set  
$$
\eta (p,t_{1}, \dots , t_{s}) = t_{1}Y_{1}(p) + \dots  + t_{s}Y_{s}(p) \ ,
$$
which makes $\eta $   a linear isometry sending $\mathbb{R}^{s}(r)$ to $\widetilde{E}^{s}(p,r)$ and depending smoothly on $p$.  The composition
$$
\exp \circ \operatorname{gr}(\gamma ) \circ  \eta 
$$
is a uniformly $C^{1}$ parameterization of the strong stable plaque $W^{s}(p,r)$.  We denote the parameterization as   $\rho (p,t)$ and define a corresponding $\mathcal{W}$-chart $\varphi  _{p}: \mathbb{R}^{s}(r) \times  \mathbb{R}^{s}(r) \rightarrow  W$ by 
$$
\varphi _{p}(x,y) = (\rho (p,x),\rho (q,y)) \quad \textrm{where} \quad q = \rho (p,x) \ .
$$
As $U$ ranges through the finite set of chart neighborhoods $U$ that cover $M$ and $p$ ranges through $U$ this gives a bundle atlas $\mathcal{A}$  for $\mathcal{W}$.  Chart transfers are uniformly $C^{1}$ since they are just related by different choices of trivializing vector fields, and one set of trivializing vector fields is related to another by orthogonal maps $\widetilde{E}^{s}(p) \rightarrow  \widetilde{E}^{s}(p)$ that depend smoothly on $p$. 

The amalgam map $a$ is defined from globally $C^{1}$ data, namely the uniformly $C^{1}$ transversals $W^{s}_{g}$,  the diffeomorphism $f$, and the uniformly $C^{1}$ holonomy maps associated to the $f$-invariant $C^{1}$ foliation $\mathcal{F}$.  Thus the fiber contractions $F = a \times  g$ and $G = g \times  a$ on $\mathcal{W}$ are uniformly $C^{1}$ when represented in the charts $\varphi _{p} \in \mathcal{A}$.  The key quantities in the hypotheses of Theorem~\ref{t.UHST} are the vertical shear   and the contraction rates of fiber versus base.  They are now easy to estimate.    

The vertical shear is uniformly $1$-bounded because $F$ and $G$ are uniformly $C^{1}$.  The fiber contraction of $F$ is $\norm{T^{s}g}$ and the base contraction is $\pmb{m}(T^{s}a)$.   Since $a$ $C^{1}$-approximates $g$ and $g$ $C^{1}$-approximates $f$ the stable bunching condition $\nu < \mu ^{\theta }$, which is  
$$
\sup_{p \in M} \frac{\norm{T^{s}_{p}f}}{(\pmb{m}(T^{s}_{p}f))^{\theta }} < 1 \ ,
$$
gives
$$
\sup_{p \in M} \frac{\norm{T^{s}_{p}g}}{(\pmb{m}(T^{s}_{p}a))^{\theta }} < 1 \quad \textrm{and} \quad  \sup_{p \in M} \frac{\norm{T^{s}_{p}a}}{(\pmb{m}(T^{s}_{p}g))^{\theta }} < 1 \ .
$$
Thus the fiber contractions of $F$ and $G$ uniformly $\theta $-dominate the base contractions, and
  Theorem~\ref{t.UHST} implies that the unique $F$- and $G$-invariant sections $\sigma $ and $\tau $  are $\theta $-H\"older.

 \section{The proof of Theorem~B}  
 \label{s.C}
 
 Theorem~B concerns   uniformly compact laminations, and the following result of  David  Epstein is used.  See also \cite{Pablo}.
 
 \begin{Thm}
 \label{t.DBAE} \cite{DBAE} 
Each leaf of a uniformly compact lamination has arbitrarily small laminated neighborhoods. 
\end{Thm}

A neighborhood of $L \in \mathcal{L}$ in $\Lambda $ is laminated if it consists of whole leaves of $\mathcal{L}$.  The idea of the proof is simple.  Let $L$ be a leaf of the uniformly compact lamination $\mathcal{L}$.  If the assertion is false there exist leaves $L_{n} \in \mathcal{L}$ containing points $p_{n}, q_{n}$ such that $p_{n}$ converges to some $p \in L$ and $q_{ n}$ converges to some $q \notin L$.   In order that $L_{n}$ leaves the neighborhood of $L$ it is necessary that $L_{n}$ ``spirals away'' from $L$, which causes it to have a large volume, contrary to uniform compactness of $\mathcal{L}$.  The details of the proof are not so simple.

A second result used in the proof of Theorem~B concerns plaque expansivity.

 \begin{Propn}
\label{p.pe}  \cite{Pablo} Plaque expansivity is implied by  uniform compactness, normal hyperbolicity, and dynamical coherence.
\end{Propn}

\begin{Rmk}
As noted above, plaque expansivity under the hypotheses of Theorem~\ref{t.IN}  is immediate.   For in the skew product case we have leaf expansivity, which implies plaque expansitivity.  It is interesting to note that in the uniformly compact case leaf expansivity can fail while    plaque expansivity remains true.  See Remark~\ref{k.Cat} in Section~\ref{s.cautionary} for further details.  
\end{Rmk}

\begin{proof}[\bf Proof of Proposition~\ref{p.pe}]    If $\rho , \sigma $ are nearby plaques of leaves $P, Q \in  \mathcal{L}$ then their local center unstable and center stable manifolds intersect in  a plaque of a third leaf, say
$$
\xi = W^{cu}(\rho , r) \cap W^{cs}(\sigma , r) 
$$ 
is a plaque of $X$.  Here we use dynamical coherence as defined in Section~\ref{s.AB} to assert that the intersection of the center unstable and center stable  plaques is the plaque of a leaf.     By Theorem~\ref{t.DBAE} we can assume $X \subset W^{cu}(P,r)$.  See Figure~\ref{f.dc}.
\begin{figure}[htbp]
\centering
\includegraphics[scale=.60]{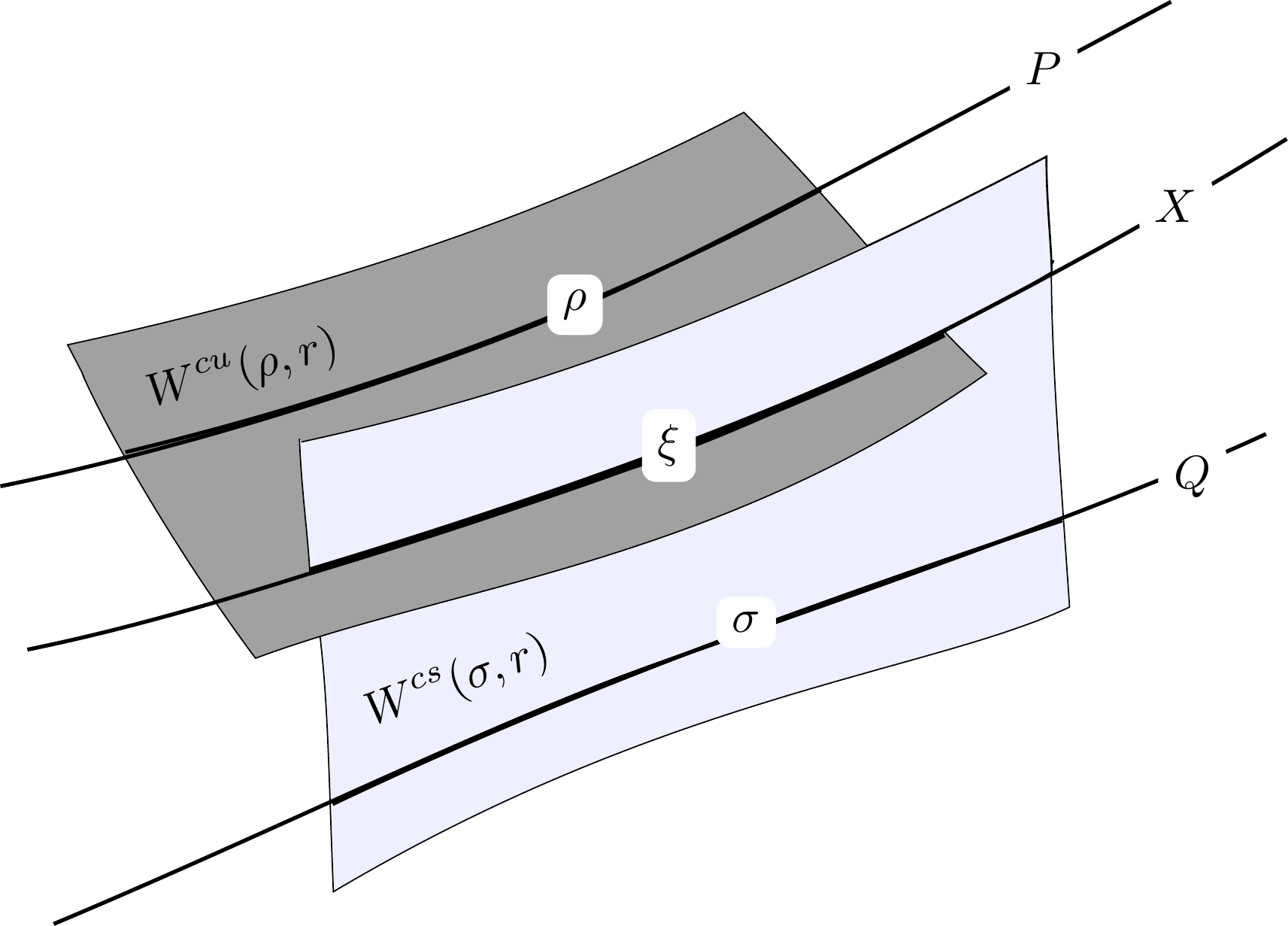}
\caption{The radius $r$ should be much less than the size of the plaques $\rho $, $\sigma $ and much greater than the distance between them although in  the figure the three quantities are not much different. The size of $\xi $ will then be on the same order as that of $\rho $ and $\sigma $.}
\label{f.dc}
\end{figure}

Consider plaque  orbits $(\rho _{k})$ and $(\sigma _{k})$ starting at $\rho $ and $\sigma $  such that the distance between $\rho _{k}$ and $\sigma _{k}$ is always much less than $r$.   Since the plaque orbits respect the lamination and $f$ preserves the laminations, the corresponding intersection plaque $\xi _{k}$ is contained in $f^{k}(X)$.  We must show that $\rho  $ and $\sigma $ overlap.

\underline{Case 1.}  $X \not=  P$.  Since $f$ overflows the family of local center unstable manifolds and $X \subset  W^{cu}(P,r)$ is compact, forward iterates  $f^{k}(X)$  are eventually pushed off $W^{cu}(f^{k}(P),r)$.   The   intersection plaque $\xi _{k}$ is contained in $f^{k}(X)$ and is therefore eventually pushed off $W^{cu}(f^{k}(P), r)$.  Correspondingly, $\xi _{k}$ and $\sigma _{k}$ become close together.  Thus, for some large   $k$ the distance between $\rho _{k}$ and $\sigma _{k}$ is on the order of $r$,   a contradiction to the assumption that the two plaque  orbits stay much closer together than $r$.

\underline{Case 2.}  $X     \not= Q$.  Using inverse iterates, we arrive at   the same contradiction.

Since the Cases 1 and 2 lead to contradictions, $P = X = Q$.  Then $\rho $ and $\sigma $ are nearby plaques on a common leaf, and such plaques always     overlap.
\end{proof}

\begin{Rmk}
We do not know whether dynamical coherence is necessary in Proposition~\ref{p.pe}.  See Remark~\ref{k.dc} in Section~\ref{s.cautionary}.
\end{Rmk}

\begin{proof}[\bf Proof of Theorem~B]  We are given a $C^{1}$  diffeomorphism $f$ which is normally hyperbolic and dynamically coherent at a uniformly compact lamination $\mathcal{L}$.  We assert that its center holonomy maps are  H\"older, and   if $g$ $C^{1}$-approximates $f$ then a canonical leaf conjugacy $\mathcal{L} \rightarrow  \mathcal{L}_{g}$ is  biH\"older.

We first examine the center holonomy maps restricted to the center unstable manifolds.  By the proof Theorem~4.3 in  \cite{PSW97}  the choice of transversals affects the H\"older constant but not the H\"older exponent of a holonomy map.  Thus  we can use the local strong unstable manifolds of $f$ (or of $g$ if $g$ perturbs $f$) as transversals.  So let $h$ be a center holonomy map $W^{u}(p,r) \rightarrow W^{u}(p^{\prime},r^{\prime})$.  It is  determined by a path $\xi : [0,1] \rightarrow  L$ from $p$ to $p^{\prime}$ in the leaf $L \in \mathcal{L}$ containing $p, p^{\prime}$.   
We claim that $h$ is $ \theta  $-H\"older when  
$$
\widehat{\nu } < \widehat{\mu }^{ \theta }  \ .
$$
To cut down on the number of hats and reciprocal hats we set
$$
\lambda = \widehat{\nu }^{-1} \quad \omega = \widehat{\mu }^{-1}   \ .
$$
Then $1 < \lambda < T^{u}f < \omega $ and we assert $h$ is $\theta $-H\"older when $\omega ^{\theta } < \lambda $.    We write
\begin{equation*}
\begin{split}
\lambda ^{k}(p) &= \lambda (f^{k-1}(p)) \cdot \lambda (f^{k-2}(p)) \cdots \lambda (f(p)) \cdot \lambda (p)
\\
\omega  ^{k}(p) &= \omega  (f^{k-1}(p)) \cdot \omega  (f^{k-2}(p)) \cdots \omega  (f(p)) \cdot \omega  (p)
\end{split}
\end{equation*}
in order to make product formulas simple. 

Fix a small $R>0$ such that for all $p \in \Lambda $, all  $x \in W^{u}(p,R)$, and all unit vectors $v \in T_{x}(W^{u}(p,R))$ we have
$$
\lambda (p) < \abs{T_{x}f(v)} < \omega (p) \ .
$$
Let $\Omega = \max_{x\in M} \norm{T_{x}f}$.  
Given $a, b \in W^{u}(p,r) \cap \Lambda $ we choose paths $\alpha, \beta $ in the leaves $A, B$ containing $a, b$ that start at $a,b$, closely shadow $\xi $, and end at points $a^{\prime}, b^{\prime} \in W^{u}(p^{\prime}, r^{\prime})$.  The center holonomy map $h$  sends $a, b$ to $a^{\prime}, b^{\prime}$.   We parameterize the paths so that $ \beta (t) \in W^{u}(\alpha  (t), R)$ for $0 \leq  t \leq  1$.  See Figure~\ref{f.holonomy6}.
\begin{figure}[htbp]
\centering
\includegraphics[scale=.60]{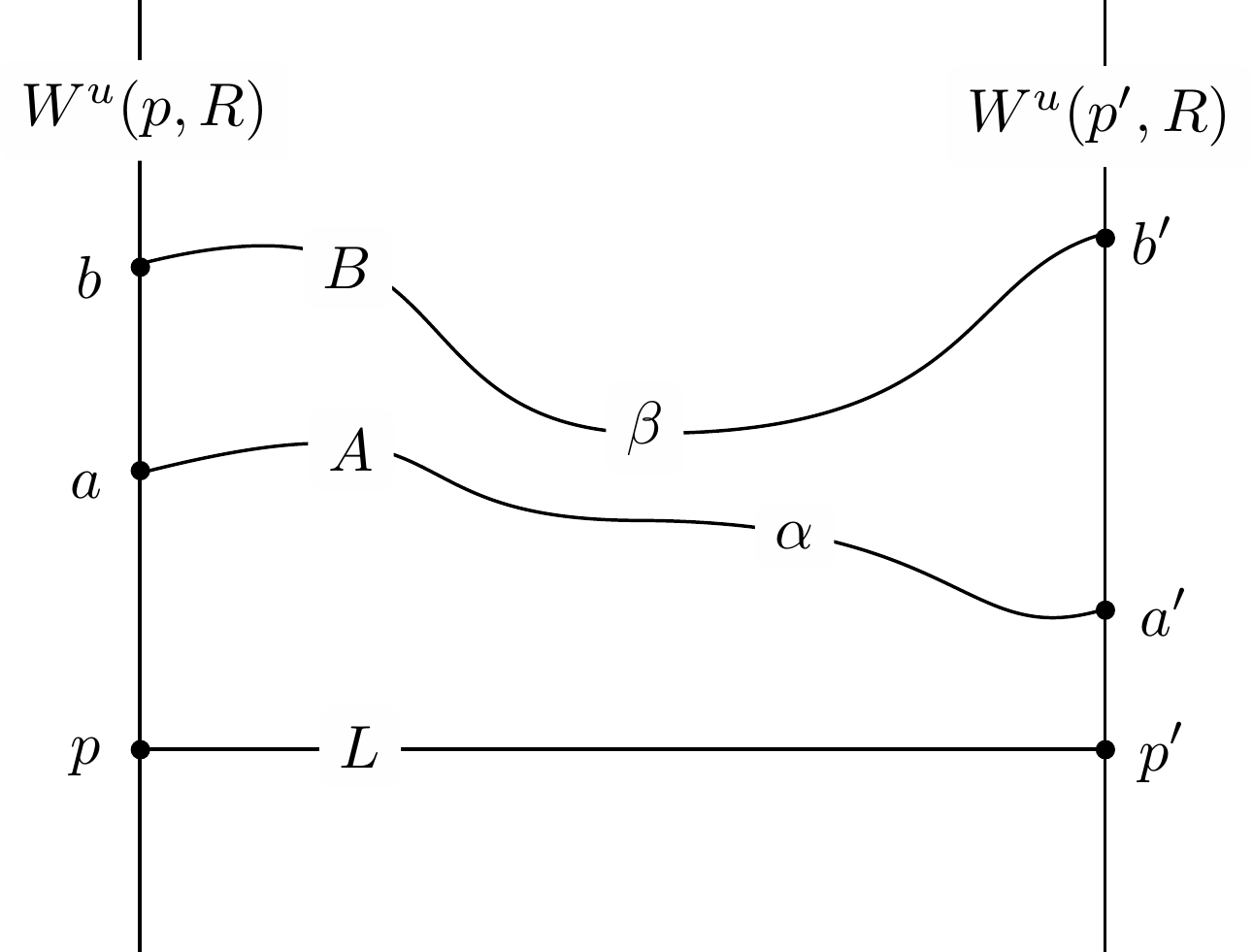}
\caption{The center holonomy inside a local center unstable manifold $W^{cu}(L, R)$.}
\label{f.holonomy6}
\end{figure}
 Here $0 < r < r^{\prime} <  R/2\Omega $ 
are sufficiently  small:  For each plaque $\rho $ in a fixed plaquation of $\mathcal{L}$, the local strong unstable manifolds $W^{u}(x,  R)$ with $x \in \rho $ give a tubular neighborhood of $\rho $ in the local center unstable manifold.       

Suppose that the center holonomy maps $h$ inside the center unstable manifolds fail    to be uniformly $\theta $-H\"older.  Then we can find such sequences $L_{n}$, $p_{n}$, $p_{n}^{\prime}$, $\xi _{n}$, $a_{n}$, $a_{n}^{\prime}$, $A_{n}$, $\alpha _{n}$, $b_{n}$, $b_{n}^{\prime}$, $B_{n}$,  $\beta _{n}$, $C_{n}$ such that 
$$
d_{n} ^{\prime} > C_{n}d_{n}^{\theta } \quad \textrm{and} \quad C_{n} \rightarrow  \infty 
$$
where $d_{n} = d(a_{n}, b_{n})$,  $d_{n}^{\prime} = d(a_{n}^{\prime}, b_{n}^{\prime})$, and distance is measured along the local strong unstable manifolds. Since $d_{n}^{\prime}$ is bounded (by $2r^{\prime}$) and $C_{n} \rightarrow  \infty$, we get $d_{n} \rightarrow  0$.   See Figure~\ref{f.holonomy7}.
\begin{figure}[htbp]
\centering
\includegraphics[scale=.60]{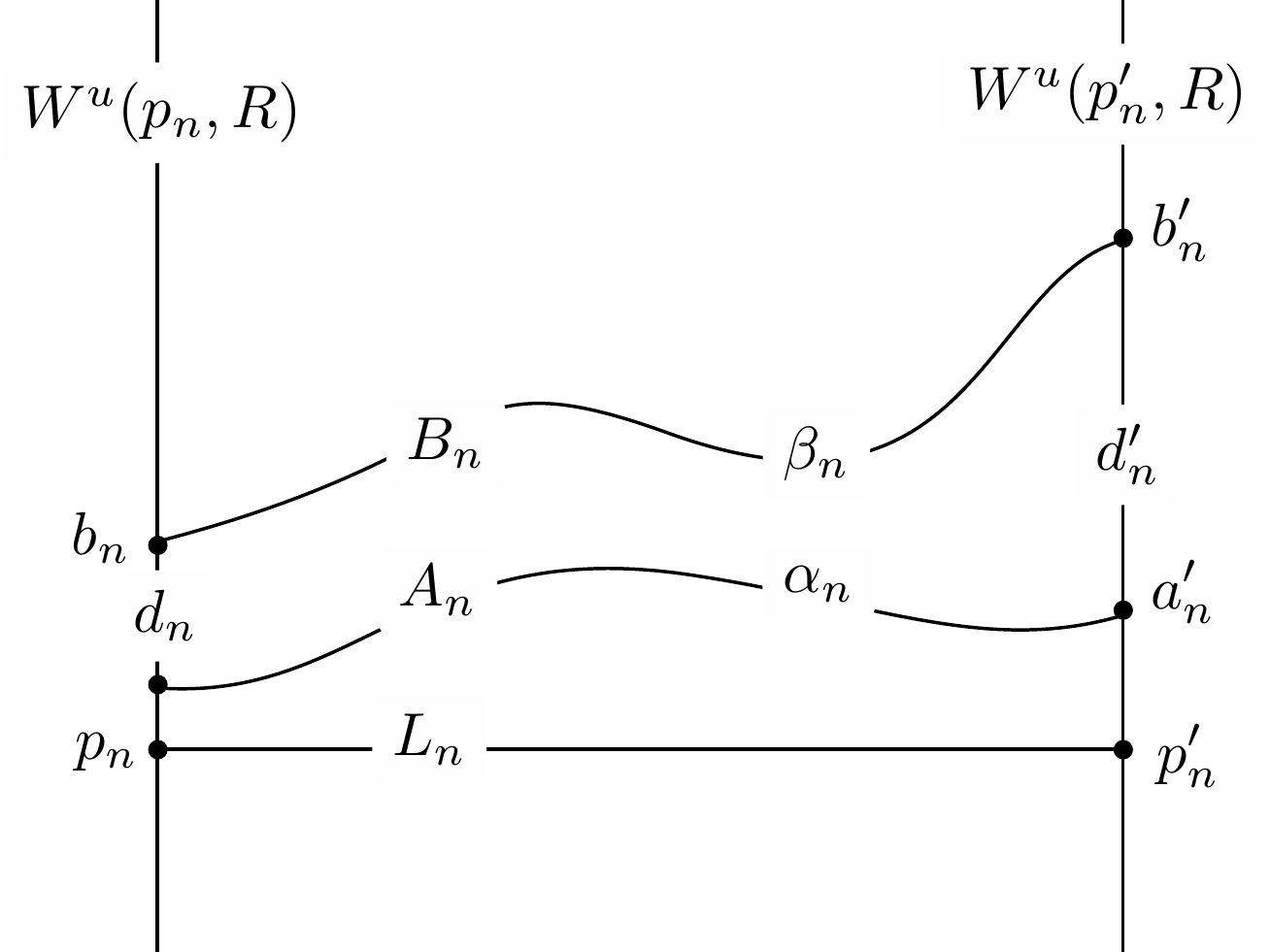}
\caption{$d_{n}$ is much less than $d_{n}^{\prime}$.}
\label{f.holonomy7}
\end{figure}
Composing everything with  $f^{k}$ produces sequences $L_{n,k}, \dots  , \beta _{n,k}$.  Set
$$
d_{n,k}(t) = d(\alpha_{n,k}(t), \beta _{n,k}(t)) \ .
$$
For each   $n$  let $k = k(n) $ be the smallest integer such that
\begin{itemize}

\item[(a)]
$0 \leq  t \leq  1$ and $0 \leq  j \leq  k $ imply 
$$
  \beta _{n,j}(t) \in W^{u}(\alpha _{n,j} (t), R) \ .
$$
\item[(b)]
There exists $t \in [0,1]$ such that $d_{n,k}(t) \geq  r$.
\end{itemize}
The original paths $\alpha_{n} , \beta_{n} $ satisfy (a),       $f$ expands the unstable manifolds by a factor between $\lambda $ and $\omega $, and $  \beta _{n}(t) \in W^{u}(\alpha _{n} (t), R)$, so $k$ exists.    
(Here we use the assumption that $\Omega r^{\prime} \leq R/2$.)   Let $t=T_{n,k}$ be the smallest $t$ such that $d (\alpha _{n,k}(t), \beta _{n,k} (t))   = r$.  We claim that for large $n$, $T_{n,k}$ exists and $0 < T_{n,k} \leq  1$.  See Figure~\ref{f.holonomy8}. \begin{figure}[htbp]
\centering
\includegraphics[scale=.60]{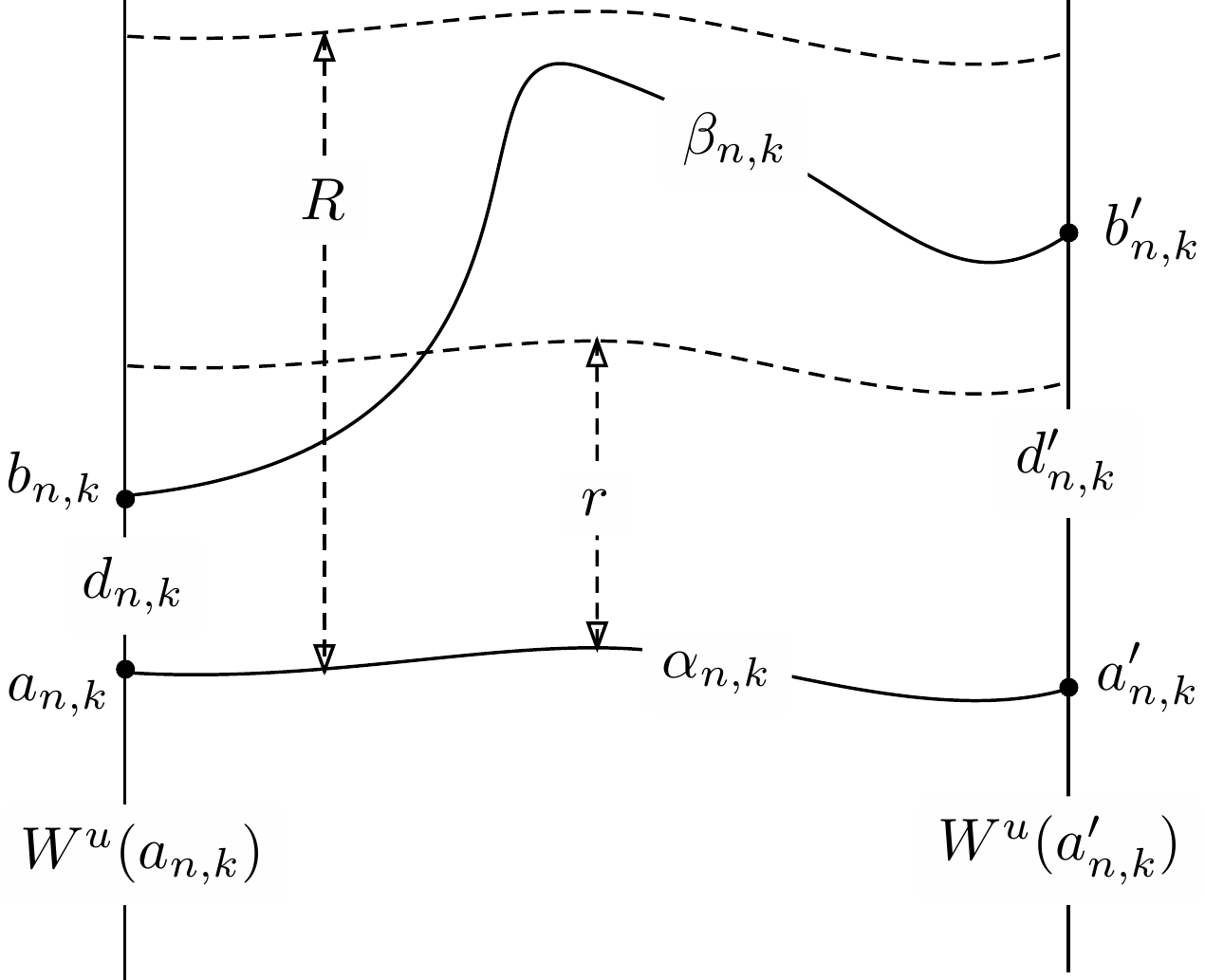}
\caption{For $k = k(n)$ we have $d(\alpha_{n,k}(t), \beta _{n,k}(t)) = r$ for some $t \in (0, 1]$.}
\label{f.holonomy8}
\end{figure}
This  is the heart of the proof.

Let $d_{n,j} = d (a_{n,j}, b_{n,j} )$ and $d_{n,j}^{\prime}  = d (a_{n,j}^{\prime}, b_{n,j}  ^{\prime} )$, where $a_{n,j} = f^{j}(a_{n})$, $a_{n,j}^{\prime} = f^{j}(a_{n}^{\prime})$, $b_{n,j}  = f^{j}(b_{n} ) $, $b_{n,j} ^{\prime}  = f^{j}(b_{n}^{\prime} )$, and $0 \leq  j \leq  k$.  
Then
$$
\lambda < \frac{d_{n,j+1}}{d_{n,j}} < \omega 
\qquad  \lambda < \frac{d_{n,j+1}^{\prime} }{d_{n,j}^{\prime} } < \omega
$$ 
An over-estimate for $d_{n,j}$ imagines the expansion occurs at the fast rate $\omega $.  This gives $d_{n,j} \leq  \omega ^{j}d_{n}$ as long as $d_{n,j} \leq  R$. Similarly, $d_{n,j}^{\prime} \geq \lambda ^{j}d_{n}^{\prime}$.   We claim that if $n$ is large and    $0 \leq  j \leq  k$ then we have $d_{n,j} < r $.   
This follows from formula manipulation.  First of all, the hypothesis   $\omega ^{\theta } < \lambda $ implies  
$$
\frac{\omega }{\lambda ^{1/\theta }} < 1 \ .
$$
Then $\lambda ^{j}d_{n}^{\prime}  \leq d_{n,j}^{\prime} \leq  R/2$ for $0 \leq  j \leq  k$ and $d_{n}^{\prime}  \geq C_{n}d_{n}^{\theta }$ imply that
\begin{equation*}
\begin{split}
d_{n,j} &\leq \omega ^{j}d_{n} \leq \omega ^{j} \Big(\frac{d_{n}^{\prime} }{C_{n}}\Big)^{1/\theta } 
\leq 
\omega ^{j} \Big(\frac{d_{n,j}^{\prime} }{\lambda ^{j}C_{n}}\Big)^{1/\theta }
\\
& \leq \Big(\frac{\omega }{\lambda ^{1/\theta }}\Big)^{j}\Big(\frac{R}{ 2C_{n}}\Big)^{1/\theta }   \ .
\end{split}
\end{equation*}
Since $R$ is fixed and   $C_{n} \rightarrow \infty$, this quantity tends to $0$ as $n\rightarrow \infty$.   Therefore, for all large $n$ and all $j$, $0 \leq  j \leq  k = k(n)$,  
$$
d_{n,j}  < r 
$$
and $d_{n,k} \rightarrow 0$ as $n \rightarrow \infty $.
 
The upshot is this: The first $k$ iterates of $f$ spread the pair $(a_{n}^{\prime}, b_{n}^{\prime} )$ apart to distance  $ \geq r$ but  don't spread the pair $(a_{n},b_{n} )$ apart much at all.  By the Intermediate Value Theorem there is a smallest $t \in (0,1]$ with $d_{n,k}(t) = r$.  This is  $ T_{n,k}  $.   
Then    we   linearly reparameterize $\alpha _{n,k}$ and $\beta _{n,k} $ by $t \rightarrow  t/T_{n, k}$ so that $T_{n,k}$ becomes $1$ and, using the same notation for the reparmeterized paths, we have  $(a_{n,k}^{\prime}, b_{n,k}^{\prime} ) =   (\alpha _{n,k}(1), \beta _{n,k} (1))$ and
$$
d(\alpha _{n,k}(1), \beta _{n,k} (1)) = r \ .
$$

Since the lamination is uniformly compact there is a subsequence having a lot of convergence.  For example, a subsequence of $f^{k(n)}(a_{n})$ converges to some $p \in \Lambda $.  Formally we should write   $f^{k(n_{\ell})}(a_{n_{\ell}})  \rightarrow  p$ as $\ell \rightarrow  \infty$ but we abbreviate it to $a_{m} \rightarrow  p$ as  $m =  (n_{\ell}, k(n_{\ell})) \rightarrow  \infty$.

Since $d(a_{m}, b_{m} ) \rightarrow 0$, $b_{m} $ also converges to $p$. Set 
$$
d_{m}(t) = d(\alpha _{m}(t), \beta _{m}(t)) \ .
$$
Then $d_{m}(0) \rightarrow  0$ and $d_{m}(1) = r$ as $m \rightarrow \infty$.  
 Let $P$ be the leaf of $\mathcal{L}$ through $p$, and let $\pi  : U \rightarrow P$ be a small $C^{1}$ tubular neighborhood of $P$ in $M$.    We   choose $\pi $ so that its fibers at $P$ are approximately parallel to $E^{us}_{P} = E^{u}_{P} \oplus E^{s}_{P}$ and have diameter $< r/2$.  (In fact, by the Whitney Extension Theorem we can find $\pi $ so that the $T_{x}(\pi ^{-1}(x)) = E^{us}_{x} $ for all $x \in P$.)

Theorem~\ref{t.DBAE} implies that $P$ has a   laminated neighborhood $N \subset \Lambda \cap U$, and $N$ is much smaller than $U$.  For each leaf $Q \subset  N$, $\pi  : Q \rightarrow  P$ is a covering map.  Since the leaves $A_{m}, B_{m}$  contain points near $P$, they are wholly contained in $N$, and they cover $P$ under $\pi $.  The points $\alpha _{m}(t), \beta _{m}(t)$ may not lie on a common $\pi $-fiber, but we can project $\beta _{m}(t)$ along the plaque of $B_{m}$ containing $\beta _{m}(t)$ to make this true.  Let $\beta _{m}^{\displaystyle *}(t)$ be the projected path and set  
$$
d_{m}^{\displaystyle *}(t)= d(\alpha _{m}(t), \beta _{m}^{\displaystyle *}(t)) \ .
$$
 Because the $\pi $-fibers are approximately tangent to $E^{us}_{P}$,   $W^{u}(\alpha _{m}(t))$ is approximately parallel to $E^{u}_{\alpha _{m}(t)}$, and  since
$$
E^{u}_{\alpha _{m}(t)} \subset E^{us}_{ \alpha _{m}(t)} \approx E^{us}_{\pi (\alpha _{m}(t))} \ ,
$$
we have  $d_{m}^{\displaystyle *}(t) \approx d_{m}(t)$ for $0 \leq  t \leq  1$.    Setting $b_{m}^{\displaystyle *} = \beta _{m}^{\displaystyle *}(1)$ gives
$$
d(a_{m}^{\prime}, b_{m}^{\displaystyle *}) = d^{\displaystyle *}_{m}(1)  \approx d_{m}(1) = d(a_{m}^{\prime}, b_{m}^{\prime}) = r
$$
which contradicts the fact that $a_{m}^{\prime}$ and $b_{m}^{\displaystyle *}$ lie in a set  of diameter $\leq  r/2$.   See Figure~\ref{f.rover2}. 
\begin{figure}[htbp]
\centering
\includegraphics[scale=.60]{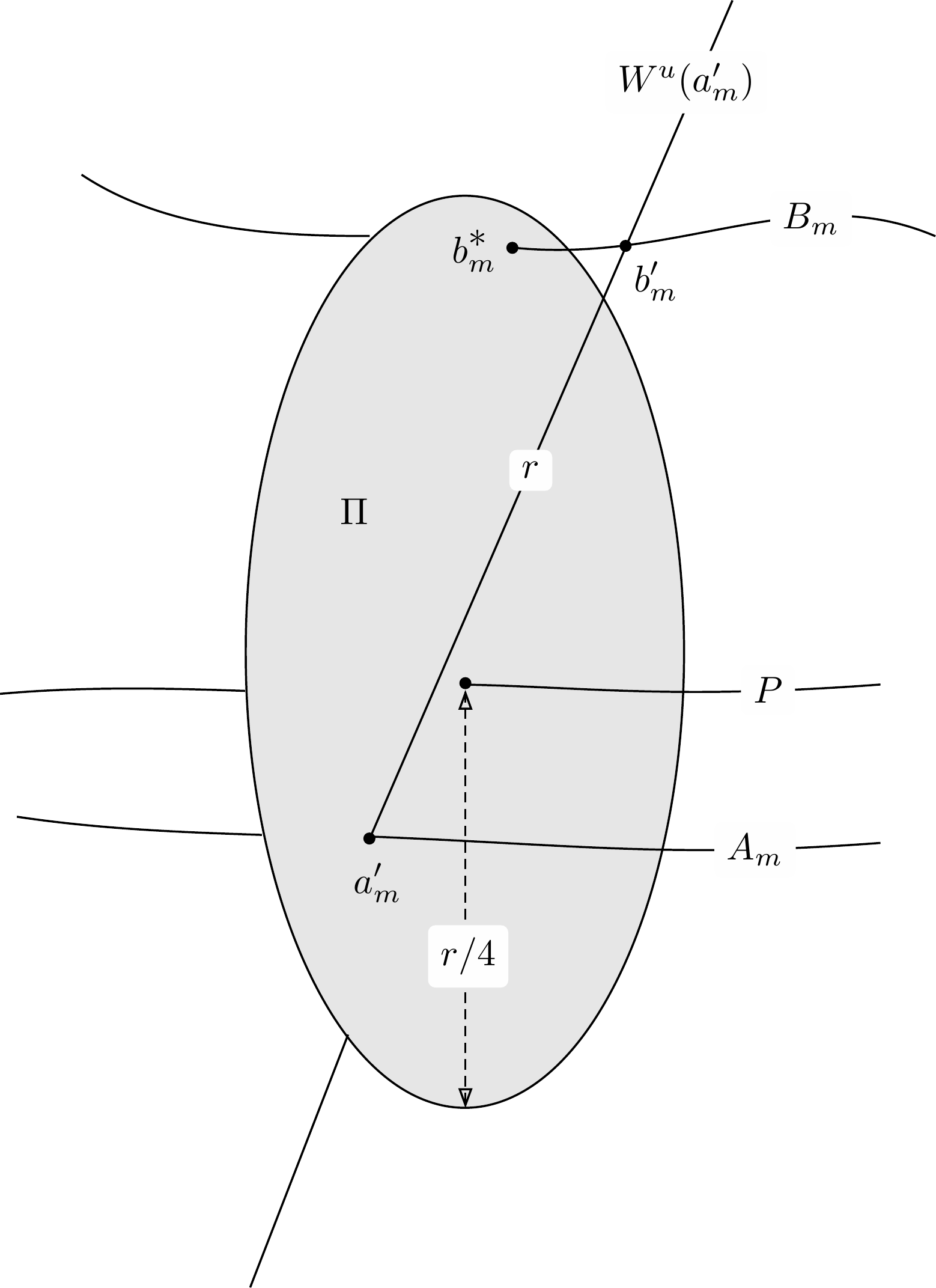}
\caption{The point $b_{m}^{\prime} =  \beta _{m}(1)$ projects along $B_{m}$ to $b_{m}^{\displaystyle *}$, so its distance to $a_{m}^{\prime}$ is approximately $r$, a contradiction to the fact that the pair $a_{m}^{\prime}, b_{m}^{\displaystyle *}$ lies in  a set $\Pi = \pi ^{-1}(\pi (a_{m}^{\prime}))$ of diameter $\leq  r/2$.} 
\label{f.rover2}
\end{figure}
Therefore the center holonomy maps along the center unstable manifolds are uniformly $\theta $-H\"older.       

Correspondingly,  the center holonomy maps along the center stable manifolds are uniformly $\theta $-H\"older   when $\nu  < \mu ^{\theta }$.  By the triangle inequality and dynamical coherence, the center holonomy  maps are uniformly $\theta $-H\"older when $\nu  < \mu ^{\theta }$ and $\widehat{\nu}  < \widehat{\mu} ^{\theta }$.    See Figure~\ref{f.triangle}.
\begin{figure}[htbp]
\centering
\includegraphics[scale=.60]{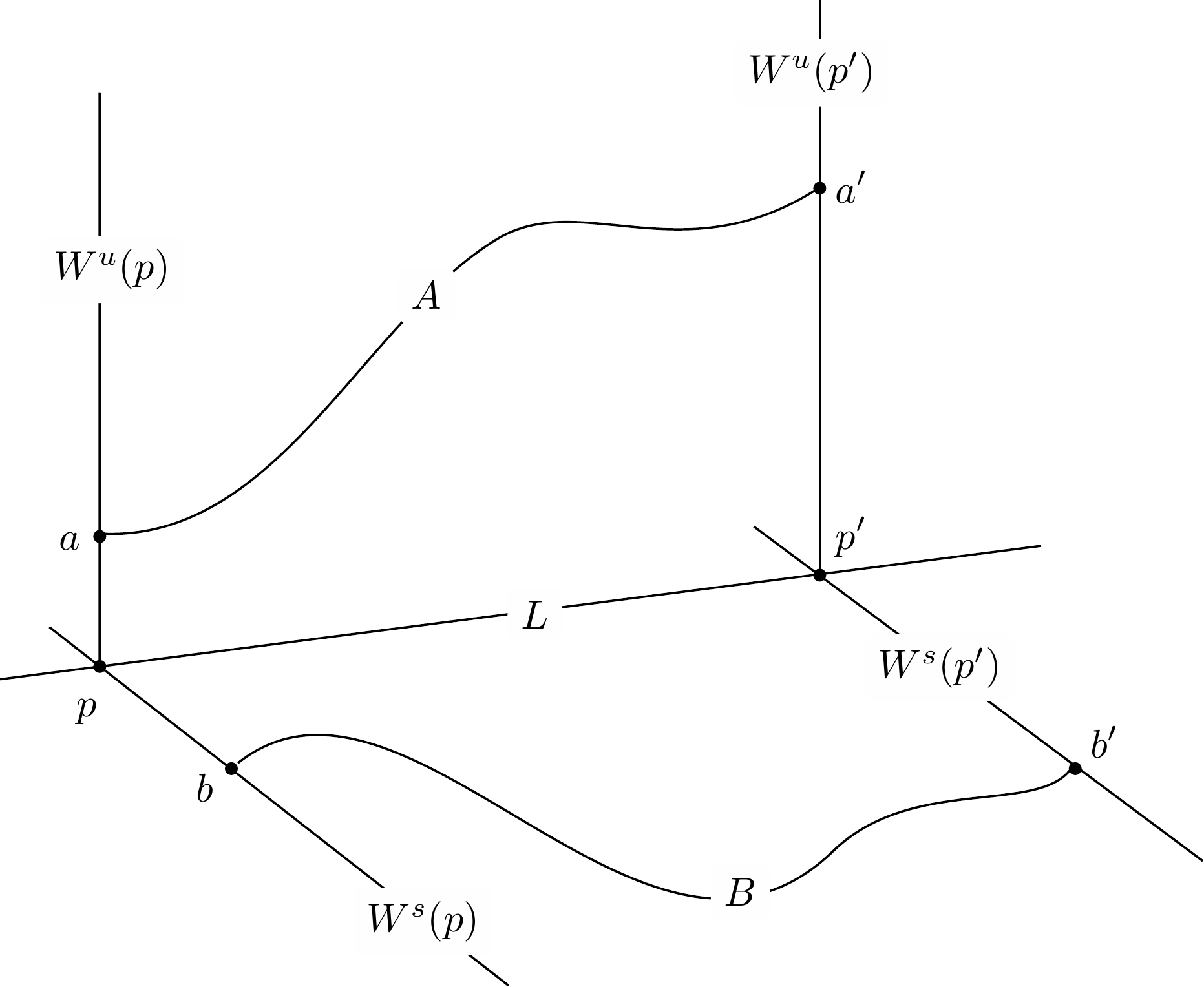}
\caption{$A, B$ are the center leaves through $a,b$.  The distance between $a^{\prime}$ and $b^{\prime}$ is no more than $d(a^{\prime}, p^{\prime}) + d(p^{\prime},b^{\prime})$ which is H\"older controlled.}
\label{f.triangle}
\end{figure}

   Proposition~\ref{p.pe} implies that $\mathcal{L}$ is plaque expansive, so a canonical leaf conjugacy $\mathfrak{h}_{g} : \mathcal{L} \rightarrow  \mathcal{L}_{g}$ exists when $g$ $C^{1}$-approximates $f$.  We claim that $\mathfrak{h}_{g}$ is $\theta  $-H\"older.

     Reverting to the suspension considerations used in the proof of Theorem~A, we have a $C^{1}$-small homotopy loop at $f$, $t \mapsto g_{t}$, where $0 \leq  t \leq  2$, $g_{0} = g_{2} = f$, and $g_{1 }= g$.  The suspension diffeomorphism $G : S^{ 1}\times  M \rightarrow  S^{ 1}\times  M$ is   defined as $G(t, x) = (t,g_{t}(x))$ where $S^{1}$ is the circle of circumference $2$.   $G$ $C^{1}$-approximates   the product diffeomorphism $F(t,x) = (t,f(x))$, which is normally hyperbolic at the product lamination $T \times \mathcal{L}$, while $G$ is normally hyperbolic at the leaf conjugate  suspension lamination $\mathcal{S}_{G}$.  The leaves of the latter are uniformly compact, so, according to what was proved above, the $\mathcal{S}_{G}$-holonomy    maps are $\theta $-H\"older.    According to Theorem~\ref{t.Amie1} and Addendum~\ref{a.lamination}, one  of these holonomy maps locally represents  a leaf conjugacy from $\mathcal{L}$ to $\mathcal{L}_{g}$, and this completes the proof.
        \end{proof}
   
      \begin{Rmk}
The geometric situation may be much more complex than a tubular neighborhood of $P$ with nearby leaves projecting diffeomorphically to $P$.  We need the H\"older estimate on transversals of uniformly positive radius. It is quite possible that $L$, $A$, and $B$ double back on themselves and each other,   repeatedly crossing a transversal.  Their various branches may lie much closer to   $\xi $ than $\alpha $ and $\beta $ do.  These other branches may shadow $\xi $ for a while and then leave its neighborhood.    The upshot is that we get $R$-sized tubular neighborhoods of the plaques but not of the leaves. See Remarks~\ref{k.Cat} and \ref{k.branches} of the next section. 
\end{Rmk}

   \section{Cautionary Remarks}
   \label{s.cautionary}
  
 \begin{Remark} 
 \label{k.translation}  
 It is natural to ask whether there is an ``Intersection Lemma" \`a la Theorem~\ref{t.intersection}
for leaf conjugacies:  For transverse foliations $\cF$ and $\cG$ intersecting in the foliation $\cH$, can one deduce from the existence of a H\"older continuous $\cF$-conjugacy and a H\"older continuous $\cG$-conjugacy the existence of a H\"older continuous $\cH$-conjugacy?  Such a general lemma would simplify considerably some of the arguments in this paper,
but it appears that such a result cannot hold in complete generality.   Here is a more detailed discussion.

In the course of proving Theorem~A, we showed directly in Proposition~\ref{p.A}
 that there exist  H\"older continuous leaf conjugacies for the center stable and center unstable
foliations. It is tempting to try to combine these leaf conjugacies to obtain directly a leaf conjugacy
for the intersection foliation $\mathcal{W}^{c} = \mathcal{W}^{cu} \cap \mathcal{W}^{cs}$.   The issue is that tubular neighborhood structures for the two conjugacies -- which consist of local unstable and stable manifolds -- in general are not jointly integrable.  They do not 
combine to give a tubular neighborhood structure for the intersection foliation.  One can choose a different
tubular neighborhood structure for $\mathcal{W}^{c}$,   one that is locally bifoliated by tubular neighborhoods for
$\cW^{cs}$ and $\cW^{cu}$, but then the question arises whether H\"older continuity of the leaf conjugacy
for one tubular neighborhood structure implies  H\"older continuity of the leaf conjugacy
for {\em every} tubular neighborhood structure.  The answer to this question, at least when
posed in the setting of abstract foliations, is ``no'' as the following example shows.

 Let $\mathcal{F}_{0}$ foliate   the strip $\RR \times [0, 1]$ in
$\RR^2$ by horizontal curves in such a way that the holonomy maps between 
vertical transversals are not H\"older continuous.  (As above, the proof of   Theorem~4.3 in \cite{PSW97} shows that the choice of transversals has no effect on H\"olderness of holonomy, so the holonomy maps with respect to all other transversals are also   non-H\"older.)  The leaf of $\cF_0$ through $(0,y)$
is given by the graph of a smooth function $x \mapsto g_0(x,y)$, where
for fixed $x \not= 0$,  the map  $y\mapsto g_0(x,y)$ is continuous but {\em not} H\"older continuous.
 Arrange as well that the top and bottom leaves of $\cF_0$
are horizontal, i.e.,
$g_0(x,0) = 0$ and $g_0(x, 1) = 1$ for all $x$.  Extend $g_{0}$ to $\mathbb{R}^{2}$ by setting
$$
g(x,y + n  ) = g_{0}(x,y) + n 
$$
when $y \in [0,1 ]$ and $n \in \mathbb{Z}$.  Let $\mathcal{F}$ be the foliation of $\mathbb{R}^{2}$ whose leaf through $(0,y)$ is $\{(x,g(x,y)) : x \in \mathbb{R}\}$.  See Figure~\ref{f.verticaltranslation}.
\begin{figure}[htbp]
\centering
\includegraphics[scale=.60]{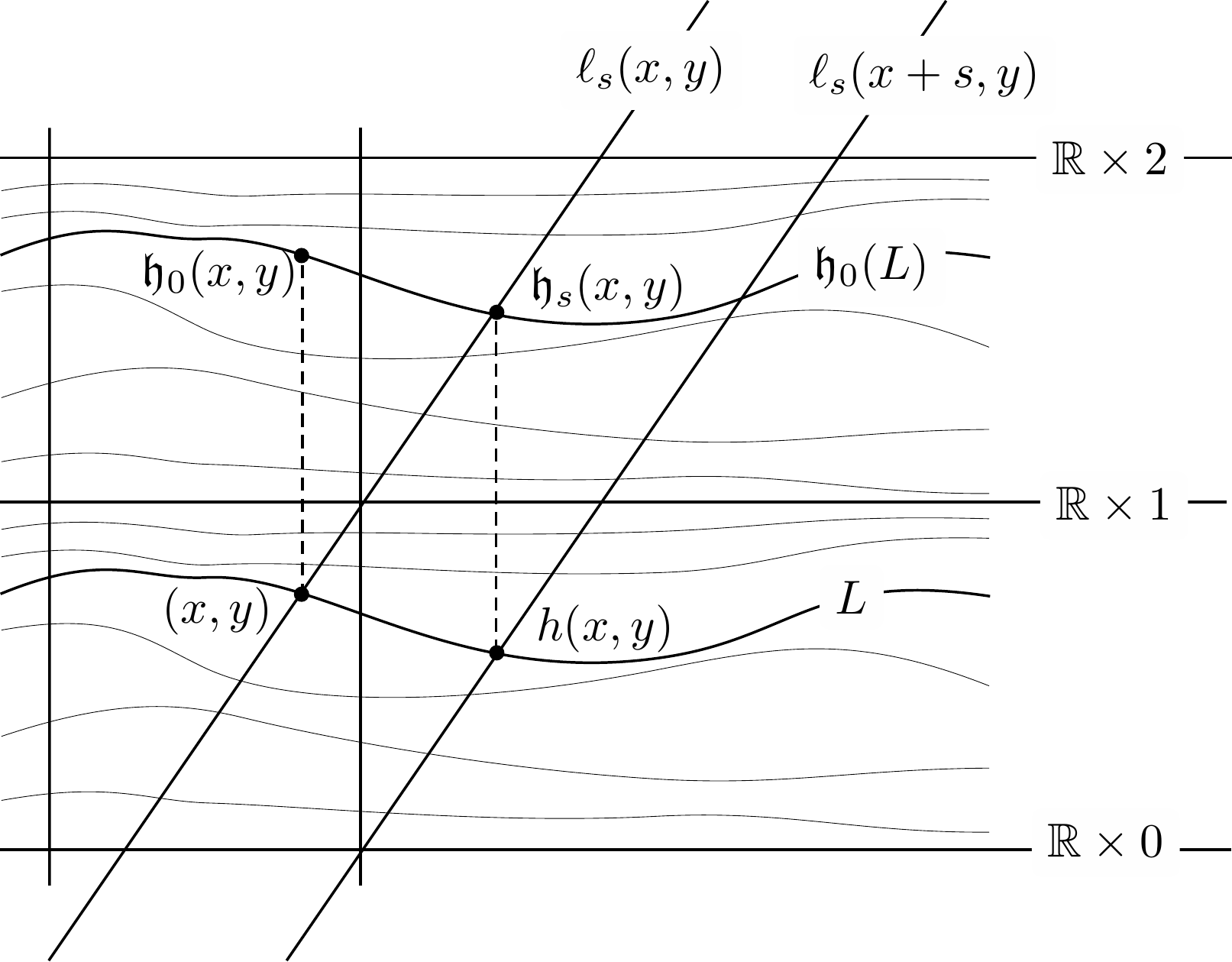}
\caption{$\mathfrak{h}_{0}$ is vertical translation by $1 $.  It commutes with the non-H\"older holonomy $h $ from $\ell_{s}(x,y)$ to $\ell_{s}(x+s, y)$.}
\label{f.verticaltranslation}
\end{figure}

Vertical translation $(x,y) \mapsto (x,y+1 )$ is a smooth leaf conjugacy $\mathfrak{h}_{0} : \mathcal{F} \rightarrow  \mathcal{F}$.  It respects the vertical normal bundle.  But if we use a different normal bundle things go bad.   Let $\mathcal{N}_{s}$ be the normal bundle whose fiber through $(x,y)$ is the line  
$$
\ell_{s}(t, x,y) = (x + st,y+t)
$$
with vertical slope $s \not= 0$.  Expressing the smooth leaf conjugacy $\mathfrak{h}_{0}$ with respect to $\mathcal{N}_{s}$ gives a leaf conjugacy $\mathfrak{h}_{s} : \mathcal{F} \rightarrow  \mathcal{F}$.  It   is a homeomorphism of $\mathbb{R}^{2}$, smooth along the leaves of $\mathcal{F}$, but it is not transversally  H\"older because
$$
\mathfrak{h}_{s}(x,y) = \mathfrak{h}_{0} \circ  h(x,y)
$$
where $h : \ell_{s}(x,y) \rightarrow \ell_{s}(x+s,y)$ is $\mathcal{F}$-holonomy.  See Figure~\ref{f.verticaltranslation}.

Hence, one needs to know something about holonomies to say anything about
leaf conjugacies.  And the smoothness of the conjugacy depends on the
choice of tubular neighborhood structure if the foliation itself is not good.  
  \end{Remark}

   \begin{Remark}
   \label{k.goodbad}
  The hypothesis in the   Intersection Lemma (Lemma~\ref{l.Hproduct}) is unnecessarily strong.  It requires all the holonomy maps of $\mathcal{F}$ and $\mathcal{G}$ to be H\"older in order that the intersection foliation $\mathcal{H} = \mathcal{F} \cap \mathcal{G}$ has  H\"older holonomy.    The following example shows we only need \emph{some} of the holonomy maps to be   H\"older.

Consider the unit cube $I^{3}$ with transverse foliations $\mathcal{F}_{0}$, $\mathcal{G}$ where the leaves of $\mathcal{F}_{0}$ are the horizontal squares $I^{2} \times  z$ and the leaves of $  \mathcal{G}$ are the vertical squares $x \times I^{2}$.  The intersection foliation   has segment leaves $x\times I \times z$.  Approximate $\mathcal{F}_{0}$ by a foliation $\mathcal{F} $ which meets every transversal $\tau  = I \times  y \times  I$ in a family of curves shown in Figure~\ref{f.T3}. Choose $\mathcal{F} $ so that its leaves are smooth but the Poincar\'e map $z \mapsto \varphi (z)$  of the flow shown on the transversals is non-H\"older.
\begin{figure}[htbp]
\centering
\includegraphics[scale=.60]{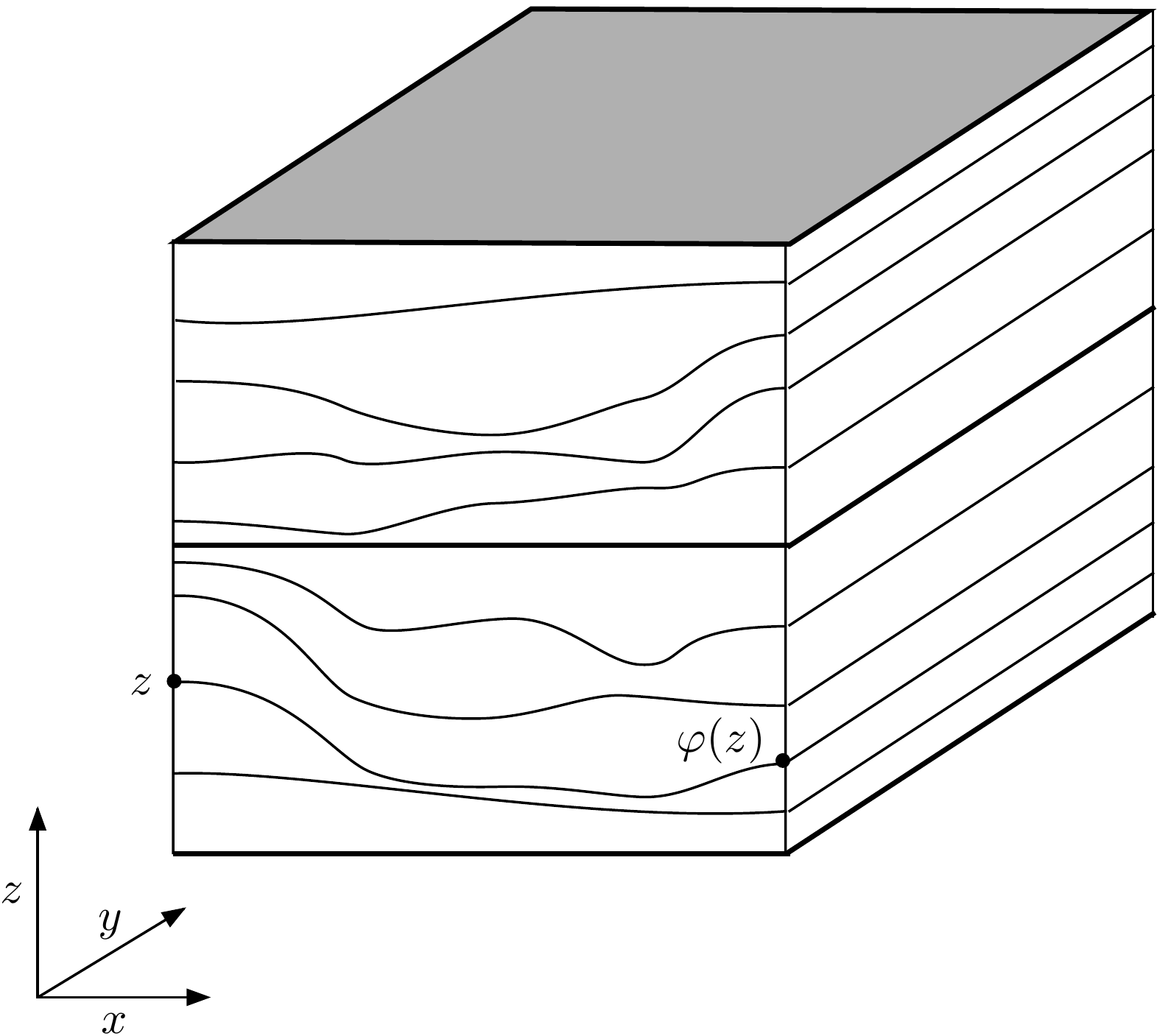}
\caption{The intersection of $ \mathcal{F} $ with the faces of the cube.}
\label{f.T3}
\end{figure}

Under the identifications that convert $I^{3}$ to the 3-torus, we get foliations $\mathcal{F} $ and $\mathcal{G}$.  The leaves of $\mathcal{G}$ are ``vertical''   2-tori.   $\mathcal{F} $ has two ``horizontal'' 2-torus leaves $A$, $B$.  They correspond to the top/bottom face of the cube and the middle slice.  The other leaves are cylinders that limit on $A$ and $B$.  The intersection foliation $\mathcal{H}$ consists of circles $x \times  S^{1} \times  z$.  The $\mathcal{H}$-holonomy is the identity map, but the $\mathcal{F} $-holonomy  includes the non-H\"older map $\varphi $.  Thus $\mathcal{F}$ and $\mathcal{G}$ can have some bad holonomy although $\mathcal{H} = \mathcal{F} \cap \mathcal{G}$ has all good holonomy.
\end{Remark}
   
  \begin{Remark}
In the proof of Theorem~B we derived a contradiction from the assumption that the $\mathcal{L}$-holonomy is not H\"older.  This involved the local center unstable and local center stable laminations.  It might have seemed more natural to prove that $\mathcal{W}^{cu}$ and $\mathcal{W}^{cs}$ are H\"older and apply the Intersection Lemma to deduce that $\mathcal{L} = \mathcal{W}^{cu} \cap \mathcal{W}^{cs}$ is H\"older.  However, $\mathcal{W}^{cu}$ and $\mathcal{W}^{cs}$ are only locally invariant and locally normally hyperbolic.  The Intersection Lemma does not directly apply in this local situation.   
\end{Remark}

 \begin{Remark}
 \label{k.Cat}
   For quite a while we were confused about the relation between leaf expansivity and plaque expansivity for normally hyperbolic foliations in the uniformly compact case.  If there is a $\delta  > 0$ such that for each pair of distinct leaves, there  is an iterate $f^{k}$ of the normally hyperbolic diffeomorphism such that the distance between the $f^{k}$-iterates of the leaves exceeds $\delta $ then $f$ is \textbf{leaf expansive}.  A skew product (with compact fiber as in Theorem~\ref{t.IN}) over a hyperbolic set has this property.  For the base map on the hyperbolic set is orbit expansive.  
     It is obvious that leaf expansivity implies plaque expansivity.  The converse, however, is false.
   
   The example occurs on a $3$-manifold.  A similar example was used for other purposes by Bonatti and Wilkinson in \cite{BonW}.    Let $M$ be   $T^{2}\times  [0,1]$ with  $(x,y,0)$ identified to $(-x,-y,1)$.  $M$ is smooth and is double covered by   the $3$-torus.  The vertical foliation $\{p \times  [0,1] : p \in T^{2}\}$ descends to a smooth, uniformly compact foliation $\mathcal{F}$ of $M$ by circles.

The standard Cat Map $f_{A} : T^{2}  \rightarrow T^{2}$ given by the matrix
 $$
 A = 
\begin{bmatrix}
   2    &  1     
\\
   1   &   1   
\end{bmatrix}
$$ 
lifts to a diffeomorphism $f : M \rightarrow  M$,
$$
f(z,t) = (f_{A}(z),t) \ ,
$$
since $A(-v) = -A(v)$ for all $v \in \mathbb{R}^{2}$.  It is normally hyperbolic and dynamically coherent  at $\mathcal{F}$. 
We claim that $f$ is plaque expansive but not leaf expansive.\footnote{$\mathcal{F}$ is a Seifert fibration whose leaf space is the 2-sphere $T^2/_{  (x,y)\sim(-x,-y)}$.  The  leaf map is a ``two pronged pseudo Anosov" map on $S^{2}$.  The leaf $L_{0}$ through the origin is a circle of length $1$.  It is fixed by $f$.  The leaves   $  L_{1}, L_{2}, L_{3}$  through  $  p_{1} = (1/2,0), p_{2} = (0,1/2), p_{3} = (1/2, 1/2)$ are also circles of length $1$.  They are permuted cyclically by $f$ as $L_{1} \rightarrow L_{2} \rightarrow  L_{3} \rightarrow  L_{1}$.   All the other leaves are circles of length $2$.  The unstable and stable manifolds of the four special leaves are M\"obius bands and the rest are cylinders.}

In \cite{HPS77} it is proved that every smooth normally hyperbolic foliation is plaque expansive, so $\mathcal{F}$ is plaque expansive.  A direct proof appears below.

To check that $f$ is not leaf expansive, consider   points $p, -p  \in T^{2}$ near the origin.  The   $\mathcal{F}$-leaf through $(p,0)$ is a circle of length $2$ in $M$ consisting of $p\times [0,1]$ and $(-p) \times  [0,1]$.   The local $f_{A}$-invariant manifolds of $p, -p$ meet at points $q, -q$ as shown in Figure~\ref{f.cat}.
\begin{figure}[htbp]
\centering
\includegraphics[scale=.60]{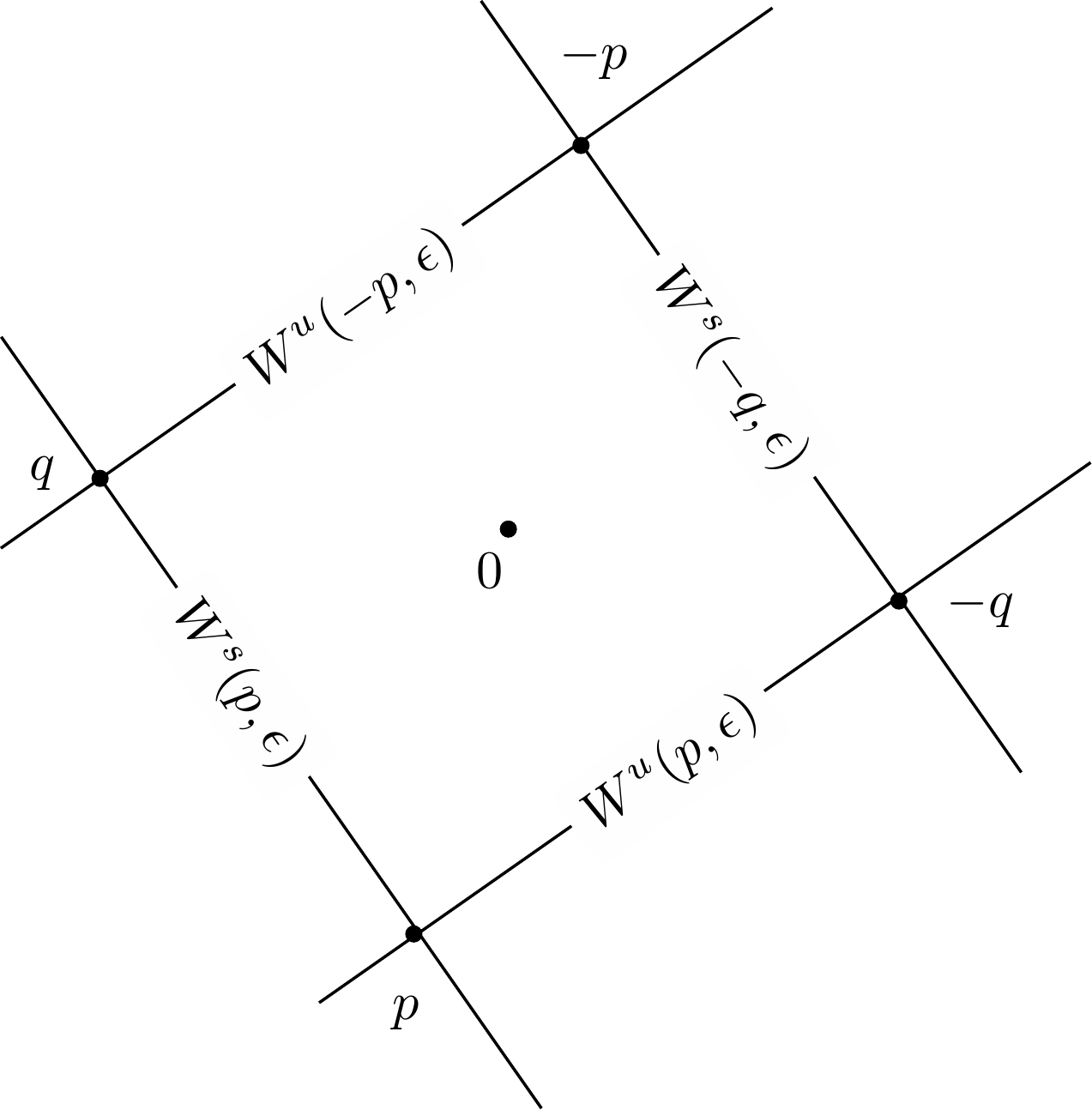}
\caption{The Cat Map separates orbits of single points but does not separate orbits of pairs of points such as $\{p,-p\}$ and $\{q,-q\}$.}
\label{f.cat}
\end{figure}The leaves $P,Q \in \mathcal{F}$ corresponding to $\{p,-p\}$   and $\{q,-q\}$ fail to separate under $f$-iteration.  For under forward iterates, $f_A^{k}(p)$ and $f_A^{k}(q)$ are asymptotic, while under reverse iteration $f_A^{k}(p)$ and $f_A^{k}(-q)$ are asymptotic.  

Here is a sketch of a direct proof that $f$ is plaque expansive.  Take non-overlapping, nearby plaques $\rho , \sigma  $   in leaves $P, Q$.  (The leaves can be equal without the plaques overlapping.)  This gives plaques 
$$
\xi = W^{cu}(\rho ,\epsilon ) \cap W^{cs}(\sigma , \epsilon ) \qquad \eta = W^{cu}(\sigma , \epsilon ) \cap W^{cs}(\rho , \epsilon ) \ .
$$   
Under forward $f$-iteration, $\xi $ and $\rho $ separate while $\eta $ and $\sigma $ are asymptotic.  Under reverse $f$-iteration it is the opposite.  See Figure~\ref{f.catplaques}.
\begin{figure}[htbp]
\centering
\includegraphics[scale=.60]{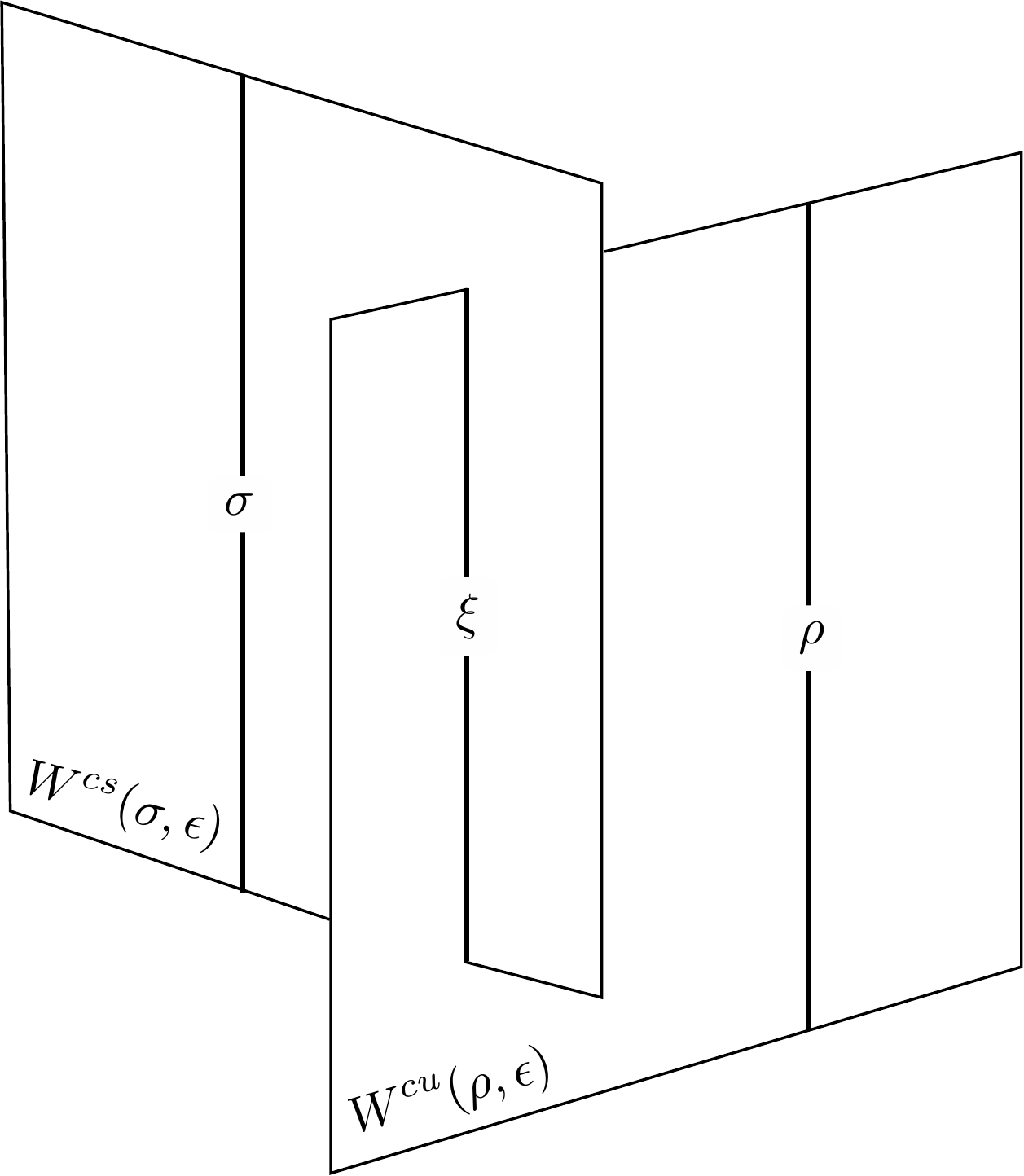}
\caption{Plaque local product structure.  The center unstable and center stable manifolds of nearby plaques intersect in plaques of approximately the same size.}
\label{f.catplaques}
\end{figure}Let $(\rho _{k})$ and $(\sigma _{k})$ be plaque orbits starting at $\rho $ and $\sigma $.  If $\sigma$ meets  $W^{cu}(\rho ,\epsilon )$ then $\sigma  \approx \xi $ and $d(\rho _{k}, \sigma _{k}) > \delta $ for a suitable $k > 0$. (By approximate equality $\sigma  \approx \xi $ we mean that $\sigma  \cap \xi $ is a plaque of approximately the same size as $\sigma $ and $\xi $.)    If $\sigma$ meets $W^{cs}(\rho ,\epsilon )$ then $\sigma  \approx \eta $ and $d(\rho _{\ell}, \sigma _{\ell}) > \delta  $ for a suitable $\ell < 0$.   Finally, if $\sigma $ meets neither $W^{cu}(\rho , \epsilon )$ nor   $W^{cs}(\rho , \epsilon )$ then $d(\rho _{k}, \sigma _{k}) > \delta$  for a suitable $k > 0$ and $d(\rho _{\ell}, \sigma _{\ell}) > \delta $ for a suitable $\ell < 0$.  

\end{Remark}

\begin{Remark}
\label{k.branches}
   A phenomenon that can occur with  normally hyperbolic, uniformly compact foliations is that the local center unstable manifold of a leaf can contain multiple branches of that leaf and other leaves.  This occurs in the previous example when the leaf lies in the local center unstable manifold of one of the special leaves  --  the circles of length $1$.   It is therefore difficult to assert in general that ``under forward $f$-iteration, the center unstable manifold is overflowing.''

\end{Remark}

   \begin{Remark}
   
   As remarked above, each leaf of a uniformly compact foliation has a tubular neighborhood, but   the radii of the tubular neighborhoods need not be bounded away from zero.  It is tempting to expect that if these radii are indeed bounded away from zero then the foliation is very nearly a skew product.

\end{Remark}

 \begin{Remark}
 \label{k.dc}
  Dynamical coherence was used in the proofs of plaque expansivity for normally hyperbolic, uniformly compact laminations (Proposition~\ref{p.pe}) and H\"olderness of the leaf conjugacy (Theorem~B).  It appears to be a challenging task to see whether dynamical coherence is really necessary.  The question is related to the concept in \cite{Pablo} of a foliation being complete.  This means that nearby leaves do not splay apart infinitely, as do the orbits of an Anosov flow.  Rather, they are somewhat parallel. Obviously, the leaves of a uniformly compact foliation have this completeness  property, but we do not know about the    intersections of their center unstable and center stable manifolds.  The dynamical coherence assumption circumvents the problem.
\end{Remark}

 \begin{Ups}
The structure of uniformly compact, normally hyperbolic foliations is yet to be well understood.
\end{Ups}

\section*{Acknowledgements}
We thank Danijela Damjanovi{\'c},  David Fisher, Andy Hammerlindl, Boris Hasselblatt, Anatole Katok, Olga Romaskevich, Ralf Spatzier, and  Andrew T\"or\"ok and the referee  for helpful conversations and comments.

   \bibliographystyle{plain}

\end{document}